\documentclass[11pt,reqno]{amsart}

\usepackage[T1]{fontenc}
\usepackage{lmodern}
\usepackage{microtype}
\usepackage{amsmath,amssymb,mathtools}
\usepackage{mathrsfs}
\usepackage{enumitem}
\usepackage{booktabs}
\usepackage{tikz-cd}
\usepackage[hidelinks]{hyperref}

\numberwithin{equation}{section}

\newtheorem{theorem}{Theorem}[section]
\newtheorem{proposition}[theorem]{Proposition}
\newtheorem{lemma}[theorem]{Lemma}
\newtheorem{corollary}[theorem]{Corollary}
\theoremstyle{definition}

\newtheorem{remark}[theorem]{Remark}

\DeclareMathOperator{\Aff}{Aff}
\DeclareMathOperator{\GL}{GL}
\DeclareMathOperator{\Area}{Area}
\DeclareMathOperator{\Aut}{Aut}
\DeclareMathOperator{\Mod}{Mod}
\DeclareMathOperator{\SO}{SO}
\DeclareMathOperator{\SL}{SL}
\DeclareMathOperator{\PSL}{PSL}
\newcommand{\C}{\mathbb C}
\newcommand{\Hh}{\mathbb H}
\newcommand{\R}{\mathbb R}
\newcommand{\Z}{\mathbb Z}
\newcommand{\cH}{\mathcal H}
\newcommand{\cQ}{\mathcal Q}
\newcommand{\DeltaTri}{\Delta}
\newcommand{\id}{\mathrm{id}}
\newcommand{\PentagonOrbit}{\mathcal O_{\mathrm{pent}}}

\title[Minimal area in genus two]
  {Minimal Hyperbolic Area of Teichm\"uller Curves in Genus Two}

\author{Xiaoyu Su}
\address{Xiaoyu Su: School of Mathematical Sciences, Beijing University of Posts and Telecommunications, Beijing 100876, P. R. China;
	Key Laboratory of Mathematics and Information Networks (Beijing University of Posts and Telecommunications), Ministry of Education, China.}
\email{xiaoyusu@bupt.edu.cn}

\author{Yumin Zhong}
\address{Yumin Zhong: School of Mathematical Sciences, Beijing University of Posts and Telecommunications, Beijing 100876, P. R. China;
	Key Laboratory of Mathematics and Information Networks (Beijing University of Posts and Telecommunications), Ministry of Education, China.}
\email{ymzhong@bupt.edu.cn}


\begin{document}

\begin{abstract}
We determine the minimum hyperbolic area of Teichm\"uller curves
arising from holomorphic quadratic differentials on closed
Riemann surfaces of genus two.  The minimum is \(3\pi/5\), and it is attained precisely by
quadratic differentials \(q=\omega^2\) for which the translation surface
\((X,\omega)\) lies in the \(\GL_2^+(\R)\)-orbit of the
double-pentagon translation surface. Equivalently, the extremal
projective Veech group is the triangle group
\(\DeltaTri(2,5,\infty)\).  The proof combines a small-area
classification of noncompact hyperbolic orbifolds with a
derivative-preserving affine descent construction for nonsquare
quadratic differentials.  The three possible nonsquare zero patterns
are then excluded by arithmetic, marked-point, and covering
obstructions.
\end{abstract}

\maketitle

\section{Introduction}

Let \(X\) be a closed Riemann surface of genus \(g\), and let
\(q\in H^0(X,\Omega_X^{\otimes2})\) be a nonzero holomorphic quadratic
differential.  Away from its zero set \(\Sigma(q)\), every point has a
\emph{natural coordinate} \(z\) in which \(q=\mathrm{d}z^2\); two such
coordinates differ by \(z\mapsto\pm z+c\).  Thus \(q\) defines a
half-translation atlas and the singular flat metric \(\lvert q\rvert\).
The points of \(\Sigma(q)\) are the singularities of this metric.  A zero
of order \(m\) has cone angle \((m+2)\pi\), and the sum of the zero
orders is \(4g-4\).  The resulting singular flat surface is denoted by
\((X,q)\) and is called a \emph{half-translation surface}.

Two half-translation surfaces \((X,q)\) and \((X',q')\) are isomorphic
if there is a biholomorphism \(\phi:X\to X'\) with \(\phi^*q'=q\).
For positive integers \(m_1,\ldots,m_n\) with
\(m_1+\cdots+m_n=4g-4\), the stratum
\(\cQ_g(m_1,\ldots,m_n)\) consists of the isomorphism classes for which
the zero orders are \(m_1,\ldots,m_n\); their order in the notation is
irrelevant.  A quadratic differential \(q\) is called a \emph{square} if
\(q=\omega^2\) for some holomorphic one-form \(\omega\) on \(X\),
and a \emph{nonsquare} otherwise.  When a stratum contains both
squares and nonsquares,
\(\cQ_g^{\mathrm{ns}}(m_1,\ldots,m_n)\) denotes its nonsquare locus. 

If \(q=\omega^2\) for a nonzero holomorphic one-form \(\omega\), the
natural coordinates are obtained by integrating \(\omega\), and their
transition maps are translations.  The pair \((X,\omega)\) is then a
\emph{translation surface}.  If the zeros of \(\omega\) have orders
\(k_1,\ldots,k_s\), then \(k_1+\cdots+k_s=2g-2\); the corresponding
stratum is denoted by \(\cH_g(k_1,\ldots,k_s)\).  A zero of order \(k\)
has cone angle \(2(k+1)\pi\).

An orientation-preserving homeomorphism \(f:X\to X\) is an
\emph{affine automorphism} of \((X,q)\) if it preserves \(\Sigma(q)\)
and, in natural coordinates, has the form
\[
  w\circ f\circ z^{-1}(\zeta)=A\zeta+b,
\]
where \(A\in\GL_2^+(\R)\), \(b\in\R^2\), and \(\GL_2^+(\R)\) is the
two-by-two orientation-preserving real general linear group.  The group
of such maps is denoted by \(\Aff^+(X,q)\).  Since an affine
automorphism preserves the total flat area, the determinant of \(A\) is \(1\), so \(A\)
lies in the two-by-two real special linear group \(\SL_2(\R)\).  A
change of natural coordinates may replace \(A\) by \(-A\), so its
class is well-defined in the projective special linear group
\(\PSL_2(\R)\).  This gives the projective derivative
homomorphism
\[
  \overline D:\Aff^+(X,q)\longrightarrow\PSL_2(\R).
\]
Its image \(\Gamma(X,q)\) is the \emph{projective Veech group} of
\((X,q)\).

For a translation surface there is no sign ambiguity.  The actual
derivative homomorphism is
\(D:\Aff^+(X,\omega)\to\SL_2(\R)\); its image
\(\SL(X,\omega)\) is the actual Veech group, and its projectivization
is denoted by \(\Gamma(X,\omega)\).

The group \(\GL_2^+(\R)\) acts on translation and half-translation
surfaces by postcomposing their flat charts with real-linear maps.  A
general element changes both the conformal structure and the
differential.  The subgroup \(\SL_2(\R)\) preserves the flat area, while positive
scalar matrices account for the overall rescaling of the
differential.
These actions preserve the zero pattern and hence
act on the strata introduced above.

Let \([X,q]\) denote the isomorphism class of \((X,q)\).  The
stabilizer of this class under the projective action is precisely
\[
  \operatorname{Stab}_{\PSL_2(\R)}[X,q]=\Gamma(X,q).
\]
After choosing a marking of \(X\) and forgetting the differential, the
\(\SL_2(\R)\)-orbit determines a map to the Teichm\"uller space
\(\mathcal T_g\).  The rotation group \(\SO(2)\) does not change the underlying
conformal structure, and the quotient of \(\SL_2(\R)\) by
\(\SO(2)\) is naturally identified with the upper half-plane
\(\Hh\).  The resulting map gives the Teichm\"uller disk
\(\mathcal D_q:\Hh\to\mathcal T_g\) generated by \((X,q)\)
\cite{EarleGardiner1997,Veech1989}.

By \cite{Veech1989,EarleGardiner1997}, the projective Veech group
\(\Gamma(X,q)\) is a discrete subgroup of \(\PSL_2(\R)\), hence a
Fuchsian group acting on the hyperbolic upper half-plane \(\Hh\).
A Fuchsian group \(\Gamma\) is a \emph{lattice} if the hyperbolic
orbifold area \(\Area(\Gamma\backslash\Hh)\) is finite.
We call \((X,q)\) a \emph{Veech surface} when \(\Gamma(X,q)\) is a
lattice; in genus at least two this lattice is necessarily nonuniform
\cite{Veech1989,EarleGardiner1997}. A Veech
surface is \emph{arithmetic} if its projective Veech group is
commensurable, up to conjugacy, with the modular group
\(\PSL_2(\Z)\), and is \emph{nonarithmetic} otherwise. When \(\Gamma(X,q)\) is a lattice, the Teichm\"uller disk descends
to an isometric orbifold immersion of
\(\Gamma(X,q)\backslash\Hh\) into \(\mathcal M_g\).
We call this immersed orbifold curve the \emph{Teichm\"uller curve}
generated by \((X,q)\), and define its hyperbolic area to be
\(\Area(\Gamma(X,q)\backslash\Hh)\). 

Earle and Gardiner considered the genus-two Riemann surface
\(X_5:y^2=x^5-1\) and the quadratic differentials
\[
  q_1=\left(\frac{\mathrm{d}x}{y}\right)^2,\qquad
  q_2=\left(\frac{x\,\mathrm{d}x}{y}\right)^2.
\]
They showed that the associated projective Veech groups are conjugate
in \(\PSL_2(\R)\) to the triangle groups
\(\DeltaTri(2,5,\infty)\) and \(\DeltaTri(5,\infty,\infty)\),
respectively; see \cite[Sections~10--11]{EarleGardiner1997}.  The
precise convention for \(\DeltaTri(p,q,r)\) is given in
Section~\ref{sec:prelim}.  The corresponding translation surfaces are the
double-pentagon and regular-decagon surfaces.  The orbifold
Gauss--Bonnet formula gives
\[
  \Area\bigl(\Gamma(X_5,q_1)\backslash\Hh\bigr)=\frac{3\pi}{5},
  \qquad
  \Area\bigl(\Gamma(X_5,q_2)\backslash\Hh\bigr)=\frac{8\pi}{5}.
\]

Motivated by these examples, Earle and Gardiner asked whether the
double-pentagon Teichm\"uller curve has the smallest possible
hyperbolic area among those arising from closed Riemann surfaces of
genus two.  More generally, they asked for the minimum in every genus
\cite[p.~166]{EarleGardiner1997}.  Equivalently, let \(m_g\) be the
infimum of \(\Area(\Gamma(X,q)\backslash\Hh)\) over all nonzero
holomorphic quadratic differentials on closed genus-\(g\) Riemann
surfaces whose projective Veech groups are lattices.  The problem is to
determine \(m_g\) and classify the equality cases.

The present paper settles the genus-two case.

\begin{theorem}\label{thm:main}
  Let \((X,q)\) be a half-translation surface of genus two. Suppose that the Veech group
  \(\Gamma(X,q)\) is a lattice.  Then
  \[
    \Area\bigl(\Gamma(X,q)\backslash\Hh\bigr)\geq\frac{3\pi}{5}.
  \]
  Equality holds if and only if \(q=\omega^2\) for a holomorphic
  one-form \(\omega\) and the translation surface \((X,\omega)\) lies
  in the \(\GL_2^+(\R)\)-orbit of the double-pentagon translation
  surface.  In this case, \(\Gamma(X,q)\) is conjugate in
  \(\PSL_2(\R)\) to \(\DeltaTri(2,5,\infty)\).
\end{theorem}

The following obstruction is a key step in the proof.

\begin{theorem}\label{thm:modular}
  Let \((X,q)\) be a half-translation surface of genus two.  Then
 the Veech group \(\Gamma(X,q)\) is not conjugate in \(\PSL_2(\R)\) to the modular
  group \(\PSL_2(\Z)\).
\end{theorem}

\begin{remark}
Several ingredients relevant to genus two are already known.  Veech's
regular-polygon construction gives the double-pentagon and
regular-decagon examples \cite{Veech1989,Veech1992}.  McMullen
classified the primitive Teichm\"uller curves generated by holomorphic
one-forms in genus two, and Bainbridge computed their Euler
characteristics \cite{McMullen2005Spin,McMullen2005,Bainbridge2007}.
M\"oller classified periodic points on primitive Veech surfaces in
genus two \cite{Moller2006}, and a broader classification of
equivariant point markings was obtained recently by Apisa
\cite{Apisa2025}.  For nonsquare quadratic differentials, the quotient
construction used below originates in work of Vasilyev
\cite{Vasilyev2005}.  We give a self-contained genus-two proof of the
precise ramification and equivariance statements required here.  We
also use Hubert--Leli\`evre's obstruction to projective elliptic
elements of order three in \(\cH_2(2)\)
\cite{HubertLelievre2006}.
\end{remark}

The proof separates the hyperbolic and geometric parts of the problem.
Orbifold Gauss--Bonnet first classifies the noncompact finite-area
hyperbolic orbifolds of area at most \(3\pi/5\).  Below the proposed
bound, only the modular group and the Hecke triangle group
\(\DeltaTri(2,4,\infty)\) can occur.  The latter is excluded on the
translation locus by a result of Hubert and Schmidt, and on the
nonsquare locus by passing to the canonical orientation double cover.
This reduces the lower bound to Theorem~\ref{thm:modular}.

The square locus consists of \(\cH_2(2)\) and \(\cH_2(1,1)\).  The
first is excluded by Hubert--Leli\`evre's order-three obstruction; the
second is excluded by comparing a square root of \(-I\) with the action
of the hyperelliptic involution on the two simple zeros.  In the
nonsquare case, the possible zero patterns are
\[
  (2,2),\qquad (2,1,1),\qquad (1,1,1,1).
\]
The Vasilyev quotient construction associates to these cases
translation surfaces of genus one, two, and three.  Every affine
automorphism descends with the same projective derivative while
preserving the relevant branch and involution data.  The stratum
\(\cQ_2(2,1,1)\) is excluded by the order-three obstruction on its
genus-two quotient.  In the other two nonsquare strata we work on the
canonical orientation cover: finite-order lifts of elliptic elements,
together with their action on the zeros and the Riemann--Hurwitz
formula, give the required contradictions.  The same method with an
order-five element excludes the remaining nonsquare equality case.  On
the square locus, an order-five elliptic element determines the
underlying Riemann surface and the relevant eigenform.

\medskip
\noindent\textbf{Organization of the paper.}
Section~\ref{sec:prelim} fixes the conventions for projective
derivative groups, triangle groups, and affine groups with regular
marked points.  Section~\ref{sec:covers} develops the canonical
orientation double cover and the Vasilyev quotient.
Section~\ref{sec:low-area} contains the low-area orbifold reduction.
Section~\ref{sec:quotients} records the quotient data for the three
nonsquare strata and proves the affine descent and arithmeticity
statements.  Section~\ref{sec:modular} proves
Theorem~\ref{thm:modular}, Section~\ref{sec:equality} treats the
equality case, and Appendix~\ref{app:quotients} proves the precise
algebraic quotient data used in the argument.

\section{Preliminaries}
\label{sec:prelim}

\subsection{Projective groups and triangle groups}

All hyperbolic area calculations in this paper use projective
derivative groups in \(\PSL_2(\R)\).  When an actual derivative group
in \(\SL_2(\R)\) is required, this will be stated explicitly.

For \(p,q,r\in\{2,3,\ldots,\infty\}\) satisfying
\(1/p+1/q+1/r<1\), with \(1/\infty=0\), let
\(T(p,q,r)\subset\Hh\) be a hyperbolic triangle with angles
\(\pi/p,\pi/q,\pi/r\).  Reflections in the sides of
\(T(p,q,r)\) generate a discrete subgroup of the full isometry group
of \(\Hh\).  Throughout this paper, \(\DeltaTri(p,q,r)\) denotes its
orientation-preserving subgroup of index two, equivalently the subgroup
consisting of even words in the side reflections.  Thus
\[
  \DeltaTri(p,q,r)<\PSL_2(\R),
\]
and a fundamental domain for \(\DeltaTri(p,q,r)\) is the union of two
adjacent copies of \(T(p,q,r)\).  An entry \(\infty\) denotes an ideal
vertex and hence a cusp in the quotient.  We call \((p,q,r)\) the
\emph{triangle type}, reserving the term \emph{orbifold signature} for
notation such as \((0;p,q;1)\).  In particular,
\[
  \DeltaTri(2,3,\infty)\cong\PSL_2(\Z).
\]

\subsection{Affine groups with regular marked points}

We shall also consider translation and half-translation surfaces
equipped with finitely many regular marked points.  Let \((S,\xi)\)
be such a surface, and let \(\Sigma(\xi)\) denote its singular set.
A \emph{marked translation surface} or \emph{marked half-translation
	surface} is a triple
\((S,\xi;P)\), where \(P\) is a finite subset of
\(S\setminus\Sigma(\xi)\).

An orientation-preserving affine automorphism of \((S,\xi;P)\) is an
affine automorphism of the underlying flat surface that preserves the
marking \(P\) setwise.  Thus
\[
\Aff^+(S,\xi;P)
:=
\{F\in\Aff^+(S,\xi):F(P)=P\},
\]
and
\[
\Gamma(S,\xi;P)
:=
\overline D\bigl(\Aff^+(S,\xi;P)\bigr).
\]
Unless explicitly stated otherwise, finite markings are always
preserved setwise rather than pointwise.  The group
\(\Gamma(S,\xi;P)\) will be called the projective Veech group of the
marked surface.

We occasionally retain an additional affine involution \(j\) as part
of the marked data.  In that case, we restrict to affine
automorphisms commuting with \(j\) and write
\[
\Aff^+(S,\xi;P,j)
:=
\{F\in\Aff^+(S,\xi;P):F\circ j=j\circ F\},
\]
and 
\[\Gamma(S,\xi;P,j)
:=
\overline D\bigl(\Aff^+(S,\xi;P,j)\bigr).
\]
When \(P=\varnothing\), we abbreviate these groups to
\(\Aff^+(S,\xi;j)\) and \(\Gamma(S,\xi;j)\), respectively.

\section{Canonical double covers and the Vasilyev quotient}
\label{sec:covers}

This section recalls the canonical orientation double cover and
establishes the functoriality under affine automorphisms that will be
used throughout the paper.  In genus two, the hyperelliptic
involution supplies a second commuting involution.  The quotient by
its distinguished lift gives the Vasilyev translation quotient.  The
stratum-by-stratum numerical data are stated in
Section~\ref{sec:quotients} and proved in
Appendix~\ref{app:quotients}.

\subsection{The canonical orientation double cover}

Let \((X,q)\) be a closed half-translation surface.  The
\emph{canonical orientation double cover} of \((X,q)\) is obtained by
normalizing the curve in the total space of the holomorphic cotangent
bundle \(\Omega_X\) defined by the equation
\(\lambda^{\otimes2}=q\); see
\cite[p.~175]{DouadyHubbard1975}.  It consists of a closed Riemann
surface \(Y\), a holomorphic double cover
\(\pi_{\mathrm{or}}:Y\to X\), and a holomorphic one-form
\(\alpha\) on \(Y\) satisfying
\[
\pi_{\mathrm{or}}^*q=\alpha^2.
\]
The cover is ramified precisely over the odd-order zeros of \(q\), and
it is connected precisely when \(q\) is not a global square.  The
nontrivial deck involution \(\iota\) exchanges the two local square
roots of \(q\), and hence
\[
\iota^*\alpha=-\alpha.
\]

We shall need the following functoriality statement.  Although the
construction of the orientation cover is standard, we include the
argument because the precise behavior of affine maps and their
projective derivatives will be used repeatedly.

\begin{lemma}
	\label{lem:lift}
	Let \((X,q)\) be a half-translation surface for which \(q\) is not a
	global square, and let
	\[
	\pi_{\mathrm{or}}:(Y,\alpha)\longrightarrow(X,q)
	\]
	be its connected canonical orientation double cover, with deck
	involution \(\iota\).  Every affine automorphism of \((X,q)\) has two
	affine lifts to \((Y,\alpha)\).  These lifts differ by \(\iota\),
	commute with \(\iota\), and have the same projective derivative as the
	original automorphism.  Consequently,
	\[
	\Gamma(X,q)
	\subseteq
	\Gamma(Y,\alpha;\iota)
	\subseteq
	\Gamma(Y,\alpha).
	\]
\end{lemma}

\begin{proof}
	Let \(\Sigma=\Sigma(q)\), and put
	\(X^\circ=X\setminus\Sigma\) and
	\(Y^\circ=Y\setminus\pi_{\mathrm{or}}^{-1}(\Sigma)\).
	The restriction
	\(\pi_{\mathrm{or}}^\circ:Y^\circ\to X^\circ\)
	is a connected unbranched double cover.
	
	For \(x\in X^\circ\), the sign holonomy of the natural-coordinate
	atlas defines a homomorphism
	\[
	\varepsilon_{q,x}:
	\pi_1(X^\circ,x)\longrightarrow\{\pm I\}.
	\]
	Analytic continuation of a local square root of \(q\) shows that, for
	\(y\in Y^\circ\) with \(\pi_{\mathrm{or}}^\circ(y)=x\),
	\[
	(\pi_{\mathrm{or}}^\circ)_*
	\pi_1(Y^\circ,y)
	=
	\ker\varepsilon_{q,x}.
	\]
	
	Let \(\phi\in\Aff^+(X,q)\), choose \(x_0\in X^\circ\), and set
	\(x_1=\phi(x_0)\).  If \(A\) represents the derivative of \(\phi\) in
	natural coordinates, then naturality of parallel transport gives
	\[
	A\,\varepsilon_{q,x_0}([\gamma])
	=
	\varepsilon_{q,x_1}(\phi_*[\gamma])\,A
	\]
	for every \([\gamma]\in\pi_1(X^\circ,x_0)\).  Since
	\(\{\pm I\}\) is central, it follows that
	\[
	\varepsilon_{q,x_1}(\phi_*[\gamma])
	=
	\varepsilon_{q,x_0}([\gamma]).
	\]
	Thus \(\phi_*\) preserves the covering subgroup, and the
	covering-space lifting criterion gives two lifts
	\(\widehat\phi^\circ:Y^\circ\to Y^\circ\), differing by \(\iota\).
	
	For either lift, the conjugate
	\(\widehat\phi^\circ\iota(\widehat\phi^\circ)^{-1}\)
	is the nontrivial deck transformation of the connected double cover.
	It therefore equals \(\iota\), so every lift commutes with \(\iota\).
	
	In the translation coordinates defined by \(\alpha\), the actual
	derivative of a lift is locally \(A\) or \(-A\).  The sign is locally
	constant and \(Y^\circ\) is connected, so the derivative is one fixed
	matrix \(B=\pm A\).  The lift and its inverse are therefore Lipschitz
	for the flat path metric.  Since \(Y\) is the metric completion of
	\(Y^\circ\), the lift extends uniquely to an affine automorphism of
	\((Y,\alpha)\).  Finally, \(B=\pm A\) implies that the lift and
	\(\phi\) have the same projective derivative.
\end{proof}

\subsection{The hyperelliptic lift and the Vasilyev quotient}
\label{sub:vas} In this subsection we describe Vasilyev's construction for nonsquare
half-translation surfaces of genus two.  Starting from the canonical
orientation double cover, the hyperelliptic involution has a
distinguished lift preserving the canonical one-form.  The quotient
by this lift is a translation surface, while the deck involution
descends to an involution of the quotient.  We first record the
genus-two fact needed in the construction.

\begin{lemma}
	\label{lem:mult}
	Let \(X\) be a closed Riemann surface of genus two, let
	\(\eta_1,\eta_2\) be a basis of the vector space of holomorphic
	one-forms on \(X\), and let \(h_X\) be the hyperelliptic involution.
	Then the following statements hold.
	\begin{enumerate}
		\renewcommand{\labelenumi}{\textup{(\roman{enumi})}}
		
		\item
		The quadratic differentials
		\(\eta_1^2,\eta_1\eta_2,\eta_2^2\) form a basis of the vector space
		of holomorphic quadratic differentials on \(X\).
		
		\item
		Every holomorphic quadratic differential \(q\) on \(X\) can be
		written as \(q=\omega_1\omega_2\) for some holomorphic one-forms
		\(\omega_1,\omega_2\) on \(X\).
		
		\item
		Every holomorphic quadratic differential \(q\) on \(X\) is invariant
		under \(h_X\).  In particular, \(h_X\) is an affine automorphism of
		the half-translation surface \((X,q)\).
	\end{enumerate}
\end{lemma}

\begin{proof}
	 Statement~\textup{(i)} is a standard fact about genus-two Riemann
	 surfaces; see
	 \cite[Chapter~III, Section~7.5, Corollaries~1--2]{FarkasKra}.
	 Statement~\textup{(iii)} follows from~\textup{(i)}, since the
	 hyperelliptic involution acts as \(-I\) on the space of holomorphic
	 one-forms.  We only prove~\textup{(ii)}.  By~\textup{(i)}, every holomorphic
	quadratic differential \(q\) can be written as
	\[
	q=a\eta_1^2+b\eta_1\eta_2+c\eta_2^2
	\]
	for some \(a,b,c\in\C\).  The corresponding homogeneous quadratic
	polynomial factors over \(\C\).  Hence there exist
	\(\alpha,\beta,\gamma,\delta\in\C\) such that
	\[
	q=
	(\alpha\eta_1+\beta\eta_2)
	(\gamma\eta_1+\delta\eta_2),
	\]
	which proves~\textup{(ii)}.
\end{proof}

Let \((X,q)\) be a half-translation surface of genus two, with \(q\)
not a global square, and let
\[
\pi_{\mathrm{or}}:(Y,\alpha)\longrightarrow(X,q)
\]
be its canonical orientation double cover.  Denote its deck
involution by \(\iota\).  Thus \(Y\) is connected,
\(\pi_{\mathrm{or}}^*q=\alpha^2\), and
\(\iota^*\alpha=-\alpha\).

Let \(h_X\) be the hyperelliptic involution of \(X\).  By
Lemma~\ref{lem:mult}\textup{(iii)}, it is an affine automorphism of
\((X,q)\).  Lemma~\ref{lem:lift} therefore gives two affine lifts of
\(h_X\) to \((Y,\alpha)\), differing by composition with \(\iota\). Let \(\widetilde h\) be either lift of \(h_X\).  Since
\(\pi_{\mathrm{or}}\circ\widetilde h
=h_X\circ\pi_{\mathrm{or}}\), we have
\[
(\widetilde h^*\alpha)^2
=
\widetilde h^*(\alpha^2)
=
\widetilde h^*\pi_{\mathrm{or}}^*q
=
\pi_{\mathrm{or}}^*h_X^*q
=
\alpha^2.
\]
Because \(Y\) is connected,
\(\widetilde h^*\alpha=\pm\alpha\).  The two lifts have opposite
signs, since they differ by \(\iota\) and
\(\iota^*\alpha=-\alpha\).  Hence exactly one lift preserves
\(\alpha\).  We denote it by \(\tau\).

The defining relation for \(\tau\) is represented by the commutative
diagram
\[
\begin{tikzcd}[column sep=large,row sep=large]
	Y \arrow[r,"\tau"] \arrow[d,"\pi_{\mathrm{or}}"']
	& Y \arrow[d,"\pi_{\mathrm{or}}"] \\
	X \arrow[r,"h_X"']
	& X
\end{tikzcd}.
\]
Thus
\(\pi_{\mathrm{or}}\circ\tau=h_X\circ\pi_{\mathrm{or}}\) and
\(\tau^*\alpha=\alpha\).

Since \(h_X^2=\operatorname{id}_X\), the map \(\tau^2\) is a deck
transformation of \(\pi_{\mathrm{or}}\).  It preserves \(\alpha\),
whereas the nontrivial deck transformation \(\iota\) sends
\(\alpha\) to \(-\alpha\).  Therefore
\(\tau^2=\operatorname{id}_Y\).  Moreover,
Lemma~\ref{lem:lift} implies that \(\tau\) commutes with \(\iota\).
We have therefore obtained two commuting holomorphic involutions on
\(Y\) satisfying
\[
\tau^*\alpha=\alpha,
\qquad
\iota^*\alpha=-\alpha.
\]


Since \(\tau\) preserves \(\alpha\), it is locally a translation in
the coordinates defined by \(\alpha\).  Let
\[
Z:=Y/\langle\tau\rangle,
\qquad
\pi_V:Y\longrightarrow Z
\]
be the quotient.  The one-form \(\alpha\) descends uniquely to a
holomorphic one-form \(\omega\) on \(Z\), characterized by
\(\pi_V^*\omega=\alpha\).

Because \(\iota\) commutes with \(\tau\), it preserves the
\(\tau\)-orbits and induces a holomorphic involution \(j\) on \(Z\).
The descent is expressed by the commutative diagram
\[
\begin{tikzcd}[column sep=large,row sep=large]
	Y \arrow[r,"\iota"] \arrow[d,"\pi_V"']
	& Y \arrow[d,"\pi_V"] \\
	Z \arrow[r,"j"']
	& Z
\end{tikzcd}.
\]
Equivalently, \(j\circ\pi_V=\pi_V\circ\iota\).  Pulling back
\(j^*\omega\) to \(Y\), we obtain
\[
\pi_V^*(j^*\omega)
=
\iota^*(\pi_V^*\omega)
=
\iota^*\alpha
=
-\alpha.
\]
The injectivity of pullback by \(\pi_V\) therefore gives
\(j^*\omega=-\omega\).

The two quotient maps from \(Y\) fit into a commutative diagram
\[
\begin{tikzcd}[column sep=large,row sep=large]
	Y \arrow[r,"\pi_{\mathrm{or}}"] \arrow[d,"\pi_V"']
	& X \arrow[d,"\rho_X"] \\
	Z \arrow[r,"\rho_Z"']
	& \mathbb P^1 ,
\end{tikzcd}
\]
where \(\mathbb P^1\) is the Riemann sphere, \(\rho_X\) is the hyperelliptic quotient and \(\rho_Z\) is the
quotient by \(j\).  Indeed, the commutativity of \(\iota\) and
\(\tau\) gives
\[
Z/\langle j\rangle
\cong
Y/\langle\iota,\tau\rangle
\cong
X/\langle h_X\rangle
\cong
\mathbb P^1.
\]

Let \(B_V\subset Z\) be the branch-value set of \(\pi_V\).  We call
\((Z,\omega)\) the \emph{Vasilyev translation quotient} of
\((X,q)\), and define the associated equivariant marked data by
\begin{equation}\label{eq:vas-data}
	\mathcal V(X,q):=(Z,\omega;B_V,j).
\end{equation}
Thus \(B_V\) is retained as a finite marked set and \(j\) as an
additional involution.

In the three nonsquare strata considered below,
Proposition~\ref{prop:quot-data} shows that every point of \(B_V\) is
regular for \(\omega\) whenever \(B_V\neq\varnothing\).  We set
\[
\Aff^+\bigl(\mathcal V(X,q)\bigr)
:=
\Aff^+(Z,\omega;B_V,j),
\qquad
\Gamma\bigl(\mathcal V(X,q)\bigr)
:=
\Gamma(Z,\omega;B_V,j).
\]
The genus of \(Z\), the ramification of \(\pi_V\), the zero divisor
of \(\omega\), and the marking \(B_V\) are determined in
Section~\ref{sec:quotients} and verified algebraically in
Appendix~\ref{app:quotients}.

\section{The low-area reduction}\label{sec:low-area}

The covering tools established in Section~\ref{sec:covers}
allow us to isolate the hyperbolic part of the argument.  We equip
\(\Hh\) with the hyperbolic metric of curvature \(-1\), so that
\[
  \Area\bigl(\PSL_2(\Z)\backslash\Hh\bigr)=\frac{\pi}{3}.
\]
We first compare covolumes under inclusions of Fuchsian groups, then
classify the possible small-area orbifolds, and finally eliminate the
order-four triangle group.

\subsection{Covolume and Fuchsian overgroups}

The first tool is elementary but will be used repeatedly.  It converts
an inclusion of Fuchsian groups into an exact comparison of quotient
areas.

\begin{lemma}\label{lem:index}
Let
\(\Gamma\subseteq\Lambda<\PSL_2(\R)\) be Fuchsian groups.  If
\(\Gamma\) has finite covolume, then
\([\Lambda:\Gamma]<\infty\), and
\begin{equation}\label{eq:index}
  \Area(\Gamma\backslash\Hh)
  =
  [\Lambda:\Gamma]\,
  \Area(\Lambda\backslash\Hh).
\end{equation}
\end{lemma}

\begin{proof}
Let \(F\) be a measurable fundamental domain for \(\Lambda\).  The
translates of \(F\) indexed by the left cosets of \(\Gamma\) in
\(\Lambda\) give pairwise disjoint copies in
\(\Gamma\backslash\Hh\).  Since \(F\) has positive area, infinitely many cosets would force
infinite covolume.  In the finite-index case, summing their areas gives
\eqref{eq:index}.
\end{proof}

\subsection{Small noncompact orbifolds}

We next determine which orientable, noncompact hyperbolic orbifolds can
have area below, or equal to, the proposed bound \(3\pi/5\).
This is a purely hyperbolic statement.

\begin{lemma}\label{lem:small-orb}
Every orientable, noncompact, finite-area hyperbolic orbifold has area
at least \(\pi/3\), with equality only for signature
\((0;2,3;1)\).

If its area is strictly smaller than \(3\pi/5\), its signature is
either
\[
  (0;2,3;1)
  \qquad\text{or}\qquad
  (0;2,4;1).
\]
Moreover, if its area is equal to \(3\pi/5\), its signature is
\((0;2,5;1)\).
\end{lemma}

\begin{proof}
Let the signature of the orbifold be
\[
  (h;m_1,\ldots,m_r;c).
\]
Here \(h\) is the genus of the underlying closed Riemann surface of the orbifold.
The integers \(m_1,\ldots,m_r\geq2\) are the orders of the elliptic
points, and \(c\) is the number of cusps.  Since the orbifold is
noncompact and has
finite hyperbolic area, it has at least one cusp; hence
\(c\geq1\).
  
The orbifold Gauss--Bonnet formula gives
\begin{equation}\label{eq:gb}
  \frac{\Area}{2\pi}
  =
  2h-2+c+\sum_{i=1}^r\left(1-\frac1{m_i}\right).
\end{equation}
Set
\[
  a:=2h-2+c+\sum_{i=1}^r\left(1-\frac1{m_i}\right).
\]
Thus \(\Area=2\pi a\), and hyperbolicity implies \(a>0\).
We classify all possibilities satisfying
\[
  0<a\leq\frac3{10},
\]
which is equivalent to \(\Area\leq3\pi/5\).

\medskip
\noindent
\textit{Step 1: the underlying closed Riemann surface has genus zero.}

Suppose that \(h\geq1\).  Since \(c\geq1\) and every elliptic
contribution is nonnegative,
\[
  a=2h-2+c+\sum_{i=1}^r\left(1-\frac1{m_i}\right)\geq2h-2+c\geq1.
\]
This contradicts \(a\leq3/10\).  Therefore \(h=0\).

\medskip
\noindent
\textit{Step 2: the orbifold has exactly one cusp.}

With \(h=0\), formula \eqref{eq:gb} becomes
\[
  a=-2+c+\sum_{i=1}^r\left(1-\frac1{m_i}\right).
\]
If \(c\geq3\), then \(a\geq-2+c\geq1\), which is impossible.
If \(c=2\), then
\[
  a=\sum_{i=1}^r\left(1-\frac1{m_i}\right).
\]
The condition \(a>0\) forces \(r\geq1\).  Every elliptic point
contributes at least \(1-1/2=1/2\), so \(a\geq1/2\), again a
contradiction.  Since \(c\geq1\), the only remaining possibility is
\(c=1\).

\medskip
\noindent
\textit{Step 3: the orbifold has exactly two elliptic points.}

We now have \(h=0\) and \(c=1\), so
\begin{equation}\label{eq:one-cusp}
  a=-1+\sum_{i=1}^r\left(1-\frac1{m_i}\right)
   =r-1-\sum_{i=1}^r\frac1{m_i}.
\end{equation}
If \(r=0\), then \(a=-1\), while if \(r=1\), then
\(a=-1/m_1<0\).  Hence \(r\geq2\).

If \(r\geq3\), then \(1/m_i\leq1/2\) for every \(i\), and
\eqref{eq:one-cusp} gives
\[
  a\geq r-1-\frac r2=\frac r2-1\geq\frac12.
\]
This is impossible.  Therefore \(r=2\).

Thus the signature has the form \((0;m,n;1)\), where we may assume
\(2\leq m\leq n\), and
\begin{equation}\label{eq:two-ell}
  a=1-\frac1m-\frac1n.
\end{equation}

\medskip
\noindent
\textit{Step 4: determine the elliptic orders.}

If \(m\geq3\), then \(n\geq 3\) and therefore
\[
  \frac1m+\frac1n\leq\frac23.
\]
Equation \eqref{eq:two-ell} would give \(a\geq1/3>3/10\).
Hence \(m=2\).  We then have
\[
  a=\frac12-\frac1n.
\]
The inequality \(a>0\) gives \(n>2\), while \(a\leq3/10\) gives
\(1/n\geq1/5\), hence \(n\leq5\).  Consequently
\[
  n\in\{3,4,5\}.
\]
The corresponding areas are
\[
\begin{array}{c|c|c}
  \text{signature} & a & \Area\\ \hline
  (0;2,3;1) & 1/6 & \pi/3\\
  (0;2,4;1) & 1/4 & \pi/2\\
  (0;2,5;1) & 3/10 & 3\pi/5.
\end{array}
\]
The three assertions of the lemma now follow directly from this table.
\end{proof}

We now explain why the signatures in Lemma~\ref{lem:small-orb}
correspond to triangle groups.  By our convention, a fundamental
domain for \(\DeltaTri(p,q,\infty)\) is the union of two adjacent
copies of the hyperbolic triangle \(T(p,q,\infty)\).  The two finite
vertices give elliptic points of orders \(p\) and \(q\), respectively,
while the ideal vertex gives a cusp.  Hence
\(\DeltaTri(p,q,\infty)\backslash\Hh\) has orbifold signature
\((0;p,q;1)\).

Conversely, an orbifold with signature \((0;p,q;1)\) has compactified
underlying Riemann surface \(\mathbb P^1\), with two cone points of
orders \(p\) and \(q\) and one cusp.  A M\"obius transformation sends
these three distinguished points to \(0\), \(1\), and \(\infty\), so
the orbifold is unique up to conformal isomorphism.  The uniqueness of
hyperbolic orbifold uniformization then implies that every Fuchsian
group with this signature is conjugate in \(\PSL_2(\R)\) to
\(\DeltaTri(p,q,\infty)\).  This identifies the signatures in
Lemma~\ref{lem:small-orb} with the triangle groups used below.

\subsection{Excluding the order-four triangle group}

The orbifold calculation leaves two possible groups below the target
area.  We now eliminate \(\DeltaTri(2,4,\infty)\), using
Lemma~\ref{lem:lift} in the nonsquare case.

\begin{proposition}\label{prop:hecke4}
The group \(\DeltaTri(2,4,\infty)\) is not the projective Veech group
of a closed half-translation surface.
\end{proposition}

\begin{proof}
In the convention of Hubert and Schmidt, the Hecke group of index
\(4\) is the Fuchsian triangle group of triangle type
\((2,4,\infty)\); in our notation it is \(\DeltaTri(2,4,\infty)\).
They use the projective Veech group in \(\PSL_2(\R)\), and their
Theorem~1 states that this group is not realizable as the Veech group
of a translation surface \cite[Theorem~1]{HubertSchmidt2001}.

Suppose now that \(q\) is nonsquare and
\(\Gamma(X,q)=\DeltaTri(2,4,\infty)\).
Let \((Y,\alpha)\) be the canonical orientation cover.  By
Lemma~\ref{lem:lift},
\[
  \Gamma(X,q)\subseteq\Gamma(Y,\alpha).
\]
If this inclusion is an equality, then
\[
  \Gamma(Y,\alpha)=\DeltaTri(2,4,\infty),
\]
contradicting the translation-surface obstruction of Hubert--Schmidt.

Suppose instead that the inclusion is strict.  Since
\(\DeltaTri(2,4,\infty)\) has covolume \(\pi/2\), its index is at
least two, and Lemma~\ref{lem:index} gives
\[
  \Area\bigl(\Gamma(Y,\alpha)\backslash\Hh\bigr)
  \leq\frac{\pi}{4}.
\]
The group \(\Gamma(Y,\alpha)\) is a Fuchsian lattice by
Lemma~\ref{lem:index}.  It contains the parabolic elements of
\(\Gamma(X,q)\), so its quotient is noncompact.  Lemma~\ref{lem:small-orb} therefore gives
\[
  \Area\bigl(\Gamma(Y,\alpha)\backslash\Hh\bigr)
  \geq\frac{\pi}{3},
\]
which is a contradiction.
\end{proof}

\begin{corollary}\label{cor:reduction}
The area bound in Theorem~\ref{thm:main} is equivalent to the
nonexistence of a holomorphic quadratic differential on a closed
Riemann surface of genus two with projective Veech group
\(\PSL_2(\Z)\).
\end{corollary}

\begin{proof}
By Lemma~\ref{lem:small-orb}, a counterexample has group
\(\DeltaTri(2,3,\infty)\) or \(\DeltaTri(2,4,\infty)\).
Proposition~\ref{prop:hecke4} excludes the second.  Conversely,
\(\PSL_2(\Z)\) has covolume \(\pi/3<3\pi/5\).
\end{proof}

Thus every possible violation of the desired lower bound has now been
reduced to the modular group \(\PSL_2(\Z)\).  Section~\ref{sec:quotients}
next applies the quotient geometry of Section~\ref{sec:covers} to
the three nonsquare strata in genus two.

\section{Quotient data in genus two and affine descent}
\label{sec:quotients}

We now apply the covering--descent construction of
Section~\ref{sec:covers} to nonsquare quadratic differentials in
genus two.  We first classify the possible zero patterns and compute the
topological and flat data of the corresponding Vasilyev quotients.  We
then prove that affine automorphisms preserve the full quotient datum
and record the arithmeticity consequences needed later.

\subsection{The nonsquare strata}\label{sub:nonsquare-strata}

The zero orders of a holomorphic quadratic differential in genus two
form a partition of \(4\).  Masur and Smillie prove that the
nonsquare singularity data \((4;-1)\) and \((3,1;-1)\) are not
realizable; see \cite[Theorem~1 and p.~293]{MasurSmillie1993}.
Consequently, the only nonsquare strata in genus two are
\[
\cQ_2^{\mathrm{ns}}(2,2),\qquad
\cQ_2(2,1,1),\qquad
\cQ_2(1,1,1,1).
\]
Equivalently, \(\cQ_2(3,1)=\varnothing\), and every quadratic
differential in \(\cQ_2(4)\) is a square.

\subsection{The quotient datum by stratum}

For a nonsquare genus-two half-translation surface \((X,q)\), let
\[
  \pi_{\mathrm{or}}:(Y,\alpha)\longrightarrow(X,q)
\]
be its canonical orientation double cover, and write
\[
  \mathcal V(X,q)=(Z,\omega;B_V,j),
  \qquad
  \pi_V:Y\longrightarrow Z,
\]
for the Vasilyev datum and quotient map constructed in
Subsection~\ref{sub:vas}.  The construction originates in Vasilyev's
work \cite{Vasilyev2005}.  The form needed here also records the
branch-value set and the residual involution.  We state the resulting
data by stratum below; Appendix~\ref{app:quotients} gives a
self-contained genus-two proof using explicit Riemann-surface models.

\begin{proposition}
	\label{prop:quot-data}
	For a closed Riemann surface \(S\), let \(g(S)\) denote its genus.
	The downstairs data have the following form in the three nonsquare
	strata:
	\begin{center}
		\small
		\begin{tabular}{ccccc}
			\toprule
			\((X,q)\) & \(g(Y)\) & \((Z,\omega)\) & \(B_V\subset Z\)
			& action of \(j\)\\
			\midrule
			\(\cQ_2^{\mathrm{ns}}(2,2)\) & \(3\) & flat torus
			& four regular points & fixed-point-free on \(B_V\)\\
			\(\cQ_2(2,1,1)\) & \(4\) & \(\cH_2(2)\)
			& two regular points & interchanges them\\
			\(\cQ_2(1,1,1,1)\) & \(5\) & \(\cH_3(2,2)\)
			& \(\varnothing\) & fixes both double zeros\\
			\bottomrule
		\end{tabular}
	\end{center}
	In particular, every point of \(B_V\) is regular for \(\omega\).
	In the genus-three case, the two double zeros of \(\omega\) are
	Weierstrass points fixed by \(j\), and \(\tau\) acts on the four
	zeros of \(\alpha\) as a product of two disjoint transpositions.
\end{proposition}

The proof is given in Appendix~\ref{app:quotients}.

The proposition gives the three concrete forms of the Vasilyev quotient
datum \(\mathcal V(X,q)\) defined in \eqref{eq:vas-data}.

\subsection{Affine functoriality}

The underlying translation quotient alone is not enough for the group
comparison.  We must show that every affine symmetry of \((X,q)\)
preserves the full quotient datum \(\mathcal V(X,q)\).  For a closed
oriented surface \(S\), let \(\Mod(S)\) denote its
orientation-preserving mapping class group.

\begin{lemma}
\label{lem:hyp}
Every element of \(\Aff^+(X,q)\) commutes with the hyperelliptic
involution \(h_X\).
\end{lemma}

\begin{proof}
The natural homomorphism
\[
  \Aff^+(X,q)\longrightarrow\Mod(X)
\]
is injective.  Indeed, an affine automorphism with nonconformal linear
part is the Teichm\"uller extremal map in its isotopy class, whereas
the identity class has extremal dilatation one.  Hence an affine
automorphism isotopic to the identity must have conformal linear part.
It is then a conformal automorphism of \(X\), and a conformal
automorphism of a closed Riemann surface of genus at least two that is
isotopic to the identity is the identity.  See
\cite{Veech1989,MollerAffineGroups}.

By the Birman--Hilden theorem, the mapping class of the hyperelliptic
involution is central in the genus-two mapping class group:
\[
  [h_X]\in Z(\Mod(X));
\]
see \cite{BirmanHilden1971,BirmanHilden1973}.  Lemma~\ref{lem:mult}
gives \(h_X^*q=q\), so \(h_X\in\Aff^+(X,q)\).  For every
\(\phi\in\Aff^+(X,q)\), centrality gives
\[
  [\phi h_X\phi^{-1}]=[h_X].
\]
Injectivity now implies \(\phi h_X\phi^{-1}=h_X\).
\end{proof}

\begin{remark}
General results on affine coverings compare Veech groups only after the
relevant ramification and branch data have been retained.  In
particular, branching over regular points may change the
commensurability class obtained after the branch values are forgotten;
see \cite{HubertSchmidt2000,HubertSchmidt2001}.  This is why the regular
branch-value set \(B_V\) is retained in \(\mathcal V(X,q)\).  The
involution \(j\) is separate equivariant structure, and is retained
because affine maps descending from \((X,q)\) must commute with it.
The proposition below gives the derivative-preserving inclusion
needed in this paper, rather than only a commensurability statement.
No converse lifting assertion is required.
\end{remark}

\begin{proposition}
	\label{prop:descent}
	Let \((X,q)\) be a genus-two half-translation surface with \(q\)
	nonsquare, and let
	\[
	\mathcal V(X,q)=(Z,\omega;B_V,j)
	\]
	be its Vasilyev datum.  Then
	\begin{equation}\label{eq:desc-inc}
		\Gamma(X,q)
		\subseteq
		\Gamma\bigl(\mathcal V(X,q)\bigr)
		\subseteq
		\Gamma(Z,\omega).
	\end{equation}
\end{proposition}

\begin{proof}
Let \(\gamma\in\Gamma(X,q)\), and choose
\(\phi\in\Aff^+(X,q)\) with \(\overline D(\phi)=\gamma\).  By
Lemma~\ref{lem:lift}, \(\phi\) has an affine lift
\(\widehat\phi\in\Aff^+(Y,\alpha;\iota)\) with the same projective
derivative.

We claim that \(\widehat\phi\) commutes with \(\tau\).  By
Lemma~\ref{lem:hyp}, \(\phi\) commutes with \(h_X\),
so
\(
\widehat\phi\tau\widehat\phi^{-1}
\)
is a lift of \(h_X\).  The two lifts of \(h_X\) are \(\tau\) and
\(\iota\tau\), with actual derivatives \(I\) and \(-I\), respectively.
If \(A=D\widehat\phi\), then
\[
  D(\widehat\phi\tau\widehat\phi^{-1})=AIA^{-1}=I.
\]
Therefore \(\widehat\phi\tau\widehat\phi^{-1}=\tau\).

It follows that \(\widehat\phi\) descends uniquely through
\(\pi_V:Y\to Z\) to an affine automorphism
\(\phi_Z\in\Aff^+(Z,\omega)\) satisfying
\[
  \phi_Z\circ\pi_V=\pi_V\circ\widehat\phi.
\]
The identity \(\pi_V^*\omega=\alpha\) implies
\(\overline D(\phi_Z)=\gamma\).  Since \(\widehat\phi\) commutes with
\(\tau\), it preserves the fixed-point set \(\operatorname{Fix}(\tau)\)
of \(\tau\), and hence
\(\phi_Z(B_V)=B_V\).  Moreover,
\[
\begin{aligned}
  \phi_Z\circ j\circ\pi_V
  &=\phi_Z\circ\pi_V\circ\iota
   =\pi_V\circ\widehat\phi\circ\iota\\
  &=\pi_V\circ\iota\circ\widehat\phi
   =j\circ\phi_Z\circ\pi_V.
\end{aligned}
\]
The surjectivity of \(\pi_V\) gives
\(\phi_Z\circ j=j\circ\phi_Z\).  Thus
\(\phi_Z\in\Aff^+(\mathcal V(X,q))\), proving the first inclusion in
\eqref{eq:desc-inc}.  The second follows from
\(
\Aff^+(\mathcal V(X,q))\subseteq\Aff^+(Z,\omega)
\).
\end{proof}

\subsection{Arithmeticity consequences}

The quotient construction separates the three nonsquare strata by the
genus of \(Z\).  We record the arithmeticity of the two lower-genus
cases.  These consequences will also eliminate two nonsquare cases in
the equality analysis.

A translation surface is called
\emph{primitive} if it admits no translation covering of degree greater
than one onto a translation surface of smaller genus.  In genus two, a
nonprimitive translation surface therefore admits a translation
covering onto a translation torus.

\begin{proposition}\label{prop:q22-arith}
Every Veech surface in \(\cQ_2^{\mathrm{ns}}(2,2)\) is arithmetic.
\end{proposition}

\begin{proof}
By Proposition~\ref{prop:quot-data}, the quotient \((Z,\omega)\)
is a translation torus.  After a linear normalization of its period
lattice,
\[
  \Gamma(Z,\omega)=\PSL_2(\Z).
\]
By Proposition~\ref{prop:descent},
\[
  \Gamma(X,q)\subseteq\Gamma(Z,\omega).
\]
If \((X,q)\) is a Veech surface, then \(\Gamma(X,q)\) has finite
covolume.  A finite-covolume Fuchsian subgroup of the lattice
\(\PSL_2(\Z)\) has finite index, by Lemma~\ref{lem:index}.
Hence \(\Gamma(X,q)\) is commensurable with \(\PSL_2(\Z)\), so
\((X,q)\) is arithmetic.
\end{proof}

\begin{proposition}\label{prop:q211-arith}
	Every Veech surface in \(\cQ_2(2,1,1)\) is arithmetic.
\end{proposition}

\begin{proof}
	Let \((X,q)\in\cQ_2(2,1,1)\) be a Veech surface.  By
	Proposition~\ref{prop:quot-data},
	\[
	(Z,\omega)\in\cH_2(2),
	\]
	and the branch-value set
	\[
	B_V=\{b_1,b_2\}
	\]
	consists of two regular points exchanged by the hyperelliptic
	involution \(j\) of \(Z\).  Since the Weierstrass points of \(Z\) are
	precisely the fixed points of \(j\), neither \(b_1\) nor \(b_2\) is a
	Weierstrass point.
	
	By Proposition~\ref{prop:descent},
	\[
	\Gamma(X,q)
	\subseteq
	\Gamma\bigl(\mathcal V(X,q)\bigr)
	\subseteq
	\Gamma(Z,\omega).
	\]
	Since \(\Gamma(X,q)\) is a lattice,
	Lemma~\ref{lem:index} shows that the other two groups are also
	lattices and that both inclusions have finite index.
	
	We next compare the corresponding affine automorphism groups.  The
	projective derivative maps give a commutative diagram
	\[
	\begin{tikzcd}[column sep=large,row sep=large]
		\Aff^+\bigl(\mathcal V(X,q)\bigr)
		\arrow[r,hook]
		\arrow[d,"\overline D"']
		&
		\Aff^+(Z,\omega)
		\arrow[d,"\overline D"]
		\\
		\Gamma\bigl(\mathcal V(X,q)\bigr)
		\arrow[r,hook,"\text{\rm finite index}"']
		&
		\Gamma(Z,\omega).
	\end{tikzcd}
	\]
	The vertical maps are surjective by definition.  The kernel of the
	right-hand vertical map is finite.  Indeed, an affine automorphism
	with trivial projective derivative has actual derivative \(I\) or
	\(-I\), and is therefore a conformal automorphism of the genus-two
	Riemann surface \(Z\).  The conformal automorphism group of \(Z\) is
	finite.  It follows that
	\[
	\left[
	\Aff^+(Z,\omega):
	\Aff^+\bigl(\mathcal V(X,q)\bigr)
	\right]
	<\infty.
	\]	
	By definition,
	\(\Aff^+\bigl(\mathcal V(X,q)\bigr)\) preserves
	\(B_V=\{b_1,b_2\}\).  Choose finitely many left-coset representatives
	\(F_1,\ldots,F_m\) for
	\(\Aff^+\bigl(\mathcal V(X,q)\bigr)\) in
	\(\Aff^+(Z,\omega)\).  Then, for \(i=1,2\),
	\[
	\Aff^+(Z,\omega)\cdot b_i
	\subseteq
	\bigcup_{k=1}^m F_k(B_V).
	\]
	Thus each \(b_i\) has finite orbit under \(\Aff^+(Z,\omega)\).
	
	Suppose, for contradiction, that \((Z,\omega)\) is nonarithmetic.
	Then it is primitive.  Indeed, a nonprimitive genus-two translation
	surface admits a translation covering onto a torus.  In the stratum
	\(\cH_2(2)\), the unique zero is the only possible ramification point,
	so the torus covering is branched over at most one point.  The
	Gutkin--Judge arithmeticity criterion would then imply that
	\((Z,\omega)\) is arithmetic, a contradiction; see
	\cite[Theorem~5.5]{GutkinJudge2000}.

	M\"oller's periodic-point theorem now applies to the primitive
	Veech surface \((Z,\omega)\): every regular point with finite orbit
	under its affine automorphism group is a Weierstrass point; see
	\cite[Theorem~5.1]{Moller2006}.  This contradicts the fact that
	\(b_1\) is regular, has finite affine orbit, and is not a Weierstrass
	point.  Hence \((Z,\omega)\) is arithmetic.
	
	Finally, Lemma~\ref{lem:index} gives
	\[
	[\Gamma(Z,\omega):\Gamma(X,q)]<\infty.
	\]
	Hence \(\Gamma(X,q)\) is commensurable with the arithmetic lattice
	\(\Gamma(Z,\omega)\), and therefore \((X,q)\) is arithmetic.
\end{proof}

The quotient construction is now in the form needed below.  The
appendix proves the three quotient data, Proposition~\ref{prop:descent}
gives the Veech-group inclusions associated with the Vasilyev quotient,
and the two lower-genus quotient cases are arithmetic.  We now use these
results to exclude the full modular group.

\section{Exclusion of the modular group}
\label{sec:modular}

The purpose of this section is to prove
Theorem~\ref{thm:modular}.  This also completes the
lower-bound argument by Corollary~\ref{cor:reduction}.
Suppose, for contradiction, that the projective Veech group of
\((X,q)\) is conjugate to \(\PSL_2(\Z)\).  After applying a suitable
element of \(\SL_2(\R)\), we may assume that
\[
  \Gamma(X,q)=\PSL_2(\Z).
\]

We first treat the square locus.  For the stratum
\(\cQ_2(2,1,1)\), the Vasilyev quotient transfers the modular symmetry
to a translation surface in \(\cH_2(2)\).  For the remaining two
nonsquare strata, we work directly on the canonical orientation cover.
The projective element of order three in the modular group is the main
obstruction; in the latter two cases we also use the kernel of the
actual derivative map and its induced action on the zeros of the
orientation-cover form.

\subsection{The square locus}

\begin{proposition}\label{prop:sq-mod}
  There is no genus-two translation surface \((X,\omega)\) such that
  the associated square half-translation surface \((X,\omega^2)\)
  has projective Veech group
  \[
    \Gamma(X,\omega^2)=\PSL_2(\Z).
  \]
\end{proposition}

\begin{proof}
  Suppose that such a translation surface \((X,\omega)\) exists, and
  set \(q=\omega^2\).  Then
  \[
    \Gamma(X,q)=\PSL_2(\Z).
  \]
  Let
  \[
    \operatorname{pr}:\SL_2(\R)\longrightarrow\PSL_2(\R)
  \]
  be the natural projection.

  We first recover the actual Veech group of the translation surface
  \((X,\omega)\).  Every affine automorphism of \((X,\omega)\) is
  affine for the half-translation structure defined by \(q\).
  Conversely, let \(f\in\Aff^+(X,q)\).  In translation coordinates
  obtained by integrating \(\omega\), the derivative of \(f\) is
  locally one of the two lifts of its projective derivative.  The
  choice of sign is locally constant on the connected surface
  \(X\setminus Z(\omega)\), and hence is constant.  Thus \(f\) is
  also affine for \((X,\omega)\), and
  \[
    \operatorname{pr}\bigl(\SL(X,\omega)\bigr)
    =\Gamma(X,q)=\PSL_2(\Z).
  \]
  The genus-two hyperelliptic involution \(h_X\) satisfies
  \(h_X^*\omega=-\omega\), so it has actual derivative \(-I\).
  Therefore
  \[
    \SL(X,\omega)
    =\operatorname{pr}^{-1}\bigl(\PSL_2(\Z)\bigr)
    =\SL_2(\Z).
  \]

  There are two possible zero patterns of \(\omega\).  Suppose first
  that \(\omega\in\cH_2(2)\).  The matrix
  \[
    \begin{pmatrix}
      0&-1\\
      1&1
    \end{pmatrix}
    \in\SL_2(\Z)
  \]
  has cube \(-I\), so its projective class is an elliptic element of
  order three.  This contradicts
  \cite[Proposition~4.5]{HubertLelievre2006}, which states that the
  projective Veech group of a translation surface in \(\cH_2(2)\)
  contains no elliptic element of order three.

  It remains to consider \(\omega\in\cH_2(1,1)\). Assume that
  \(\{P_1,P_2\}\) is the zero set of \(\omega\).
  Since \(h_X^*\omega=-\omega\), the hyperelliptic involution
  preserves this two-point set.  Neither zero is fixed by \(h_X\).
  Indeed, near a fixed point choose a coordinate \(z\) in which
  \(h_X(z)=-z\), and write \(\omega=f(z)\,\mathrm{d}z\).  The identity
  \(h_X^*\omega=-\omega\) gives \(f(-z)=f(z)\), so a zero fixed by
  \(h_X\) has even order.  Since \(P_1\) and \(P_2\) are simple,
  \[
    h_X(P_1)=P_2,
    \qquad
    h_X(P_2)=P_1.
  \]

  Let
  \[
    K:=\ker\left(
      D:\Aff^+(X,\omega)\longrightarrow\SL_2(\R)
    \right).
  \]
  We claim that every element of \(K\) fixes \(P_1\) and \(P_2\)
  individually.  Suppose that a nontrivial \(k\in K\) exchanges them.
  Since \(Dk=I\), the map \(k\) is locally a translation at every
  regular point and has no regular fixed point.  It fixes neither zero,
  so it acts freely on \(X\).  As a conformal automorphism of a
  genus-two Riemann surface, \(k\) has finite order, say \(n\geq2\).
  The unramified quotient would then satisfy
  \[
    2=2g(X)-2=n\bigl(2g(X/\langle k\rangle)-2\bigr),
  \]
  which is impossible.  This proves the claim.

  Now choose \(F\in\Aff^+(X,\omega)\) such that
  \[
    DF=
    \begin{pmatrix}
      0&-1\\
      1&0
    \end{pmatrix}.
  \]
  Such an affine automorphism exists because
  \(\SL(X,\omega)=\SL_2(\Z)\).  The map \(F\) preserves the set
  \(\{P_1,P_2\}\), so the square of its induced permutation is the
  identity.  Hence
  \[
    F^2(P_i)=P_i,
    \qquad i=1,2.
  \]
  On the other hand,
  \[
    D(F^2\circ h_X)
    =
    \begin{pmatrix}
      0&-1\\
      1&0
    \end{pmatrix}^{\!2}
    (-I)
    =I.
  \]
  Thus \(F^2\circ h_X\in K\), and the claim says that this map fixes
  both zeros individually.  But \(h_X\) exchanges the two zeros while
  \(F^2\) fixes them individually, so \(F^2\circ h_X\) exchanges them.
  This is a contradiction.
\end{proof}

\subsection{Actual derivatives on the orientation cover}

From now on, suppose that \(q\) is nonsquare and let
\[
  \pi_{\mathrm{or}}:(Y,\alpha)\longrightarrow(X,q)
\]
be its canonical orientation cover, with deck involution \(\iota\).
Recall that
\[
  \Aff^+(Y,\alpha;\iota)
  =
  \{F\in\Aff^+(Y,\alpha):F\circ\iota=\iota\circ F\}.
\]
The lifting and descent statements of
Lemma~\ref{lem:lift} give
\[
  D\bigl(\Aff^+(Y,\alpha;\iota)\bigr)
  =
  \operatorname{pr}^{-1}\bigl(\Gamma(X,q)\bigr),
\]
where
\[
  \operatorname{pr}:\SL_2(\R)\longrightarrow\PSL_2(\R)
\]
is the natural projection.  Indeed, an affine automorphism commuting
with \(\iota\) descends to \((X,q)\).  Conversely, every affine
automorphism of \((X,q)\) has two lifts to \(Y\), differing by
\(\iota\).  If one lift has derivative \(A\), then the other has
derivative \(-A\), since \(D\iota=-I\).  Thus both elements of
\(\operatorname{pr}^{-1}([A])=\{A,-A\}\) occur.

Let \(\tau\) be the lift of the hyperelliptic involution satisfying
\(\tau^*\alpha=\alpha\).  The proof of
Proposition~\ref{prop:descent} shows that every element of
\(\Aff^+(Y,\alpha;\iota)\) commutes with \(\tau\).  We shall use the following elementary
information about the kernel of the actual derivative map.

\begin{lemma}\label{lem:kernel}
  Let
  \[
    K_Y:=\ker\left(
      D:\Aff^+(Y,\alpha;\iota)\longrightarrow\SL_2(\R)
    \right).
  \]
  Then:
  \begin{enumerate}[label=\textup{(\roman*)}]
    \item if \(q\in\cQ_2^{\mathrm{ns}}(2,2)\), then
    \(|K_Y|\in\{2,4,8\}\);
    \item if \(q\in\cQ_2(1,1,1,1)\), then
    \(|K_Y|\in\{2,4,6,12\}\).
  \end{enumerate}
\end{lemma}

\begin{proof}
	Let \(\Sigma_\alpha\) be the zero set of \(\alpha\).
	For each \(P\in\Sigma_\alpha\), let \(m_P\) be the order of \(P\).
	Choose a local coordinate \(u\) centred at \(P\) such that
	\(\alpha=u^{m_P}\,\mathrm{d}u\).
	The associated local translation coordinate is
	\[
	z=\frac{u^{m_P+1}}{m_P+1}.
	\]
	A positive horizontal separatrix germ issuing from \(P\) is a germ
	of an arc starting at \(P\) whose image under \(z\) is the positive
	real ray.  There are exactly \(m_P+1\) such germs at \(P\).  Although
	the coordinate \(u\) is not unique, a different normalized choice
	merely permutes these \(m_P+1\) germs.
	
	Let
	\[
	\mathscr S^+(\alpha)
	:=
	\coprod_{P\in\Sigma_\alpha}
	\mathscr S_P^+(\alpha)
	\]
	denote the finite set of all positive horizontal separatrix germs
	issuing from the zeros of \(\alpha\).  Since every element of
	\(K_Y\) has actual derivative \(I\), it preserves \(\alpha\), and
	hence \(K_Y\) acts on \(\mathscr S^+(\alpha)\).
	
	Suppose that \(F\in K_Y\) fixes a separatrix germ
	\(\mathfrak s\in\mathscr S_P^+(\alpha)\), in the sense that
	\(F(P)=P\) and
	\(F(\mathfrak s)=\mathfrak s\)
	as germs of oriented arcs issuing from \(P\).  Choose a local
	coordinate \(u\) centred at \(P\) such that
	\(\alpha=u^{m_P}\,\mathrm{d}u\).
	Since \(F\in K_Y\), we have \(F^*\alpha=\alpha\).  Therefore
	\[
	F(u)^{m_P}F'(u)\,\mathrm{d}u
	=
	u^{m_P}\,\mathrm{d}u,
	\]
	and hence
	\[
	\mathrm{d}\left(\frac{F(u)^{m_P+1}}{m_P+1}\right)
	=
	\mathrm{d}\left(\frac{u^{m_P+1}}{m_P+1}\right).
	\]
	Since \(F(P)=P\), the integration constant is zero, and thus
	\[
	F(u)^{m_P+1}=u^{m_P+1}.
	\]
	It follows that
	\(F(u)=\zeta u\)
	near \(P\), for some \((m_P+1)\)-st root of unity \(\zeta\).
	
	In this coordinate, the \(m_P+1\) positive horizontal separatrix
	germs issuing from \(P\) are represented by the rays
	\[
	\arg u=\frac{2\pi k}{m_P+1},
	\qquad
	k=0,\ldots,m_P.
	\]
	Multiplication by \(\zeta\) permutes these germs.  Since \(F\) fixes
	the specified germ \(\mathfrak s\), we must have \(\zeta=1\).
	Consequently, \(F\) is the identity in a neighbourhood of \(P\), and
	hence
	\(F=\id_Y\).
	This proves that the action of \(K_Y\) on
	\(\mathscr S^+(\alpha)\) is free.
	
	For every
	\(\mathfrak s\in\mathscr S^+(\alpha)\), the orbit map
	\[
	K_Y\longrightarrow K_Y\cdot\mathfrak s,
	\qquad
	F\longmapsto F(\mathfrak s),
	\]
	is bijective, since the stabilizer of \(\mathfrak s\) is trivial.
	Thus every \(K_Y\)-orbit in \(\mathscr S^+(\alpha)\) has cardinality
	\(|K_Y|\).  Since \(\mathscr S^+(\alpha)\) is the disjoint union of
	its \(K_Y\)-orbits, it follows that
	\[
	|K_Y|\mid|\mathscr S^+(\alpha)|.
	\]
	
	Moreover, the involution \(\tau\) belongs to \(K_Y\), because
	\(\tau^*\alpha=\alpha\), and \(\tau\neq\id_Y\).  Hence
	\[
	2\mid |K_Y|.
	\]
	
	If \(q\in\cQ_2^{\mathrm{ns}}(2,2)\), then \(\alpha\) has four
	simple zeros.  Each of them has two positive horizontal
	separatrix germs, so
	\[
	|\mathscr S^+(\alpha)|=4\cdot2=8.
	\]
	Consequently,
	\[
	|K_Y|\mid8
	\qquad\text{and}\qquad
	2\mid|K_Y|,
	\]
	which gives
	\[
	|K_Y|\in\{2,4,8\}.
	\]
	
	If \(q\in\cQ_2(1,1,1,1)\), then \(\alpha\) has four double zeros.
	Each of them has three positive horizontal separatrix germs, so
	\[
	|\mathscr S^+(\alpha)|=4\cdot3=12.
	\]
	Consequently,
	\[
	|K_Y|\mid12
	\qquad\text{and}\qquad
	2\mid|K_Y|,
	\]
	which gives
	\[
	|K_Y|\in\{2,4,6,12\}.
	\]
\end{proof}

\subsection{The stratum \texorpdfstring{\(\cQ_2(2,1,1)\)}{Q(2,1,1)}}

\begin{proposition}\label{prop:q211-mod}
  There is no half-translation surface
  \((X,q)\in\cQ_2(2,1,1)\) such that
  \[
    \Gamma(X,q)=\PSL_2(\Z).
  \]
\end{proposition}

\begin{proof}
  Suppose that such a half-translation surface \((X,q)\) exists.
  By Proposition~\ref{prop:quot-data}, the Vasilyev
  quotient satisfies \((Z,\omega)\in\cH_2(2)\).  Moreover,
  Proposition~\ref{prop:descent} gives
  \[
    \PSL_2(\Z)=\Gamma(X,q)\subseteq\Gamma(Z,\omega).
  \]
  Thus \(\Gamma(Z,\omega)\) contains an elliptic element of projective
  order three.  This contradicts Hubert--Leli\`evre's obstruction
  \cite[Proposition~4.5]{HubertLelievre2006}.
\end{proof}

\subsection{The stratum
  \texorpdfstring{\(\cQ_2^{\mathrm{ns}}(2,2)\)}{Qns(2,2)}}

\begin{proposition}\label{prop:q22-mod}
  There is no half-translation surface
  \((X,q)\in\cQ_2^{\mathrm{ns}}(2,2)\) such that
  \[
    \Gamma(X,q)=\PSL_2(\Z).
  \]
\end{proposition}

\begin{proof}
  Suppose that such a half-translation surface \((X,q)\) exists.
  The lifting and descent relation above gives
  \[
    D\bigl(\Aff^+(Y,\alpha;\iota)\bigr)=\SL_2(\Z).
  \]
    Let
  \[
  A=
  \begin{pmatrix}
  	0&1\\
  	-1&-1
  \end{pmatrix}.
  \]
  Then \(A^3=I\) and \(A\neq I\), so
  \(\langle A\rangle\cong\Z/3\Z\).
  The restriction of the derivative homomorphism to
  \[
  \widetilde C
  :=
  D^{-1}(\langle A\rangle)
  \subseteq
  \Aff^+(Y,\alpha;\iota)
  \]
  gives a short exact sequence
  \[
  1
  \longrightarrow K_Y
  \longrightarrow \widetilde C
  \xrightarrow{\,D\,}
  \langle A\rangle
  \longrightarrow 1,
  \]
  where \[
  K_Y
  =
  \ker\left(
  D:\Aff^+(Y,\alpha;\iota)
  \longrightarrow\SL_2(\R)
  \right).
  \]
  
  By Lemma~\ref{lem:kernel}\textup{(i)}, \(K_Y\) is a finite
  \(2\)-group.  It follows from the exact sequence that
  \(\widetilde C\) is finite and
  \[
  |\widetilde C|
  =
  |K_Y|\,|\langle A\rangle|
  =
  3|K_Y|.
  \]
  Hence Cauchy's theorem gives an element
  \(F\in\widetilde C\) of order three.  Since \(K_Y\) is a
  \(2\)-group, it contains no element of order three.  Therefore
  \(F\notin K_Y\), and hence
  \(DF\neq I\).
  Since \(DF\in\langle A\rangle\), we have
  \(DF=A\) or \(DF=A^2\).  Replacing \(F\) by \(F^2\) in the
  latter case, we may assume that
  \(DF=A\).
  
 Let
 \[
 \Sigma_\alpha=\{P_1,P_2,P_3,P_4\}
 \]
 be the zero set of \(\alpha\).  Both \(F\) and \(\iota\) preserve
 \(\Sigma_\alpha\): the map \(F\) is an affine automorphism of
 \((Y,\alpha)\), while \(\iota^*\alpha=-\alpha\).  Since \(F\) has
 order three, its induced permutation of the four-element set
 \(\Sigma_\alpha\) has order dividing three.  Hence this permutation
 is either the identity or a three-cycle together with one fixed
 point.
 
 Suppose that the induced permutation is nontrivial.  Then there is
 a unique zero \(P\in\Sigma_\alpha\) fixed by \(F\).  Since
 \(F\circ\iota=\iota\circ F\), we have
 \[
 F(\iota(P))
 =
 \iota(F(P))
 =
 \iota(P).
 \]
 Thus \(\iota(P)\) is also a zero fixed by \(F\).  By the uniqueness
 of \(P\), it follows that
 \[
 \iota(P)=P.
 \]
 This contradicts the fact that \(\iota\) is fixed-point-free.
 Therefore the induced permutation of \(\Sigma_\alpha\) is trivial,
 and \(F\) fixes all four zeros of \(\alpha\).

  Let \(r\) be the number of fixed points of \(F\), and let \(g_0\) be the genus of
  \(Y/\langle F\rangle\).  Riemann--Hurwitz gives
  \[
    4
    =2g(Y)-2
    =3(2g_0-2)+2r.
  \]
  Since \(r\geq4\), this forces
  \[
    g_0=0
    \ \text{and}\
    r=5.
  \]
  On the other hand, \(F\) commutes with \(\iota\), and \(\iota\)
  acts without fixed points on \(Y\).  It therefore acts freely on
  the fixed point set of \(F\), so \(r\) must be even.  This contradiction proves the
  proposition.
\end{proof}

\subsection{The stratum
  \texorpdfstring{\(\cQ_2(1,1,1,1)\)}{Q(1,1,1,1)}}

\begin{proposition}\label{prop:q1111-mod}
  There is no half-translation surface
  \((X,q)\in\cQ_2(1,1,1,1)\) such that
  \[
    \Gamma(X,q)=\PSL_2(\Z).
  \]
\end{proposition}

\begin{proof}
	Suppose that such a half-translation surface \((X,q)\) exists.
	The lifting and descent relation above gives
	\[
	D\bigl(\Aff^+(Y,\alpha;\iota)\bigr)=\SL_2(\Z).
	\]
	Let
	\[
	A=
	\begin{pmatrix}
		0&1\\
		-1&-1
	\end{pmatrix}.
	\]
	Then \(A^3=I\) and \(A\neq I\), so
	\(
	\langle A\rangle\cong\Z/3\Z
	\).
	Set
	\[
	\widetilde C
	:=
	D^{-1}(\langle A\rangle)
	\subseteq
	\Aff^+(Y,\alpha;\iota).
	\]
	Recall that
	\[
	K_Y
	=
	\ker\left(
	D:\Aff^+(Y,\alpha;\iota)
	\longrightarrow\SL_2(\R)
	\right).
	\]
	Restriction of the derivative homomorphism gives a short exact
	sequence
	\[
	1
	\longrightarrow K_Y
	\longrightarrow \widetilde C
	\xrightarrow{\,D\,}
	\langle A\rangle
	\longrightarrow 1.
	\]
	By Lemma~\ref{lem:kernel}\textup{(ii)},
	\[
	|K_Y|\in\{2,4,6,12\}.
	\]
	Moreover,
	\[
	\widetilde C/K_Y
	\cong
	\langle A\rangle
	\cong
	\Z/3\Z,
	\]
	and hence
	\[
	|\widetilde C|=3|K_Y|
	\in\{6,12,18,36\}.
	\]
	Let \(P\) be a Sylow \(3\)-subgroup of \(\widetilde C\).
	It follows that
	\[
	|P|\in\{3,9\}.
	\]
	
Let
\[
\Sigma_\alpha=\{P_1,P_2,P_3,P_4\}
\]
be the set of the four double zeros of \(\alpha\), and let
	\[
	\rho:P\longrightarrow S_4
	\]
	be the permutation representation induced by the action of \(P\)
	on \(\Sigma_\alpha\).  Every element of \(P\) commutes with
	\(\tau\).  Moreover, Proposition~\ref{prop:quot-data} shows that the
	permutation of \(\Sigma_\alpha\) induced by \(\tau\), which we
	denote by \(t\), is a product of two disjoint transpositions.
	Consequently,
	\[
	\rho(P)\subseteq C_{S_4}(t).
	\]
	The centralizer \(C_{S_4}(t)\) has order eight, so
	\(|\rho(P)|\) divides eight.  On the other hand, \(\rho(P)\) is a
	\(3\)-group.  Therefore \(|\rho(P)|=1\), and every element of \(P\)
	fixes each of the four zeros of \(\alpha\).
	
	We next show that \(D(P)\) is nontrivial.  If
	\(|K_Y|\in\{2,4\}\), then \(|P|=3\), whereas if
	\(|K_Y|\in\{6,12\}\), then \(|P|=9\).  In either case,
	\[
	|P|\nmid |K_Y|.
	\]
	Therefore \(P\not\subseteq K_Y\) by Lagrange's theorem, and hence
	\(D(P)\) is nontrivial.  Since
	\(D(P)\subseteq\langle A\rangle\)
	and \(\langle A\rangle\) has prime order three, it follows that
	\[
	D(P)=\langle A\rangle.
	\]
	Choose \(F\in P\) such that
	\(DF=A\).
	Since \(F\neq\id_Y\) and \(F\in P\), the order of \(F\) is either
	three or nine.
	
	Suppose first that \(F\) has order three.  After replacing
	\((Y,\alpha)\) by an \(\SL_2(\R)\)-equivalent translation surface,
	we may assume that \(F\) is conformal.  Since its derivative is a
	nontrivial element of order three, there is a primitive cube root of
	unity \(\zeta\) such that
	\(
	F^*\alpha=\zeta\alpha
	\).
	
	Let \(P_0\) be one of the double zeros of \(\alpha\).  Since \(F\)
	fixes \(P_0\), choose a local coordinate \(u\) centred at \(P_0\)
	such that
	\(\alpha=u^2\,\mathrm{d}u
	\).
	Write
	\[
	F(u)=\lambda u+O(u^2),
	\qquad \lambda\neq0.
	\]
	Comparing the lowest-order terms in \(F^*\alpha=\zeta\alpha\)
	gives
	\(\lambda^3=\zeta\). On the other hand, \(F^3=\id_Y\) implies
	\(
	\lambda^3=1
	\),
	contradicting \(\zeta\neq1\).
	
	Suppose now that \(F\) has order nine.  Let \(r\) be the number
	of fixed points of \(F\), and let \(g_0\) be the genus of
	\(Y/\langle F\rangle\).  Since \(F\) fixes each of the four zeros
	of \(\alpha\), we have \(r\geq4\).

	Every fixed point of \(F\) is totally ramified in the quotient map
	\[
	Y\longrightarrow Y/\langle F\rangle
	\]
	and hence contributes \(9-1=8\) to the ramification term.  The
	Riemann--Hurwitz formula therefore gives
	\[
	8
	=
	2g(Y)-2
	\geq
	9(2g_0-2)+8r
	\geq
	9(2g_0-2)+32.
	\]
	Since \(g_0\geq0\), the last expression is at least
	\[
	-18+32=14,
	\]
	a contradiction.

	This contradiction proves the proposition.
\end{proof}

\begin{proof}[Proof of Theorem~\ref{thm:modular}]
  Suppose that \(\Gamma(X,q)\) is conjugate to \(\PSL_2(\Z)\).
  After applying an element of \(\SL_2(\R)\), we may assume that
  \[
    \Gamma(X,q)=\PSL_2(\Z).
  \]
  If \(q\) is a square, this is excluded by
  Proposition~\ref{prop:sq-mod}.  If \(q\) is nonsquare, then the classification recalled in
  Subsection~\ref{sub:nonsquare-strata} places \((X,q)\) in one of
  \[
    \cQ_2^{\mathrm{ns}}(2,2),
    \qquad
    \cQ_2(2,1,1),
    \qquad
    \cQ_2(1,1,1,1).
  \]
  These three cases are excluded by
  Propositions~\ref{prop:q22-mod},~\ref{prop:q211-mod},
  and~\ref{prop:q1111-mod}, respectively.
\end{proof}

\begin{corollary}\label{cor:bound}
  Let \((X,q)\) be a half-translation surface of genus two.  If its
  projective Veech group is a lattice, then
  \[
    \Area\bigl(\Gamma(X,q)\backslash\Hh\bigr)
    \geq \frac{3\pi}{5}.
  \]
\end{corollary}

\begin{proof}
  This follows from Corollary~\ref{cor:reduction} and
  Theorem~\ref{thm:modular}.
\end{proof}

\section{The equality case}
\label{sec:equality}

We now classify the Teichm\"uller curves attaining the minimum.
Suppose that
\[
\Area\bigl(\Gamma(X,q)\backslash\Hh\bigr)
=\frac{3\pi}{5}.
\]
By Lemma~\ref{lem:small-orb}, the projective Veech group
\(\Gamma(X,q)\) is conjugate in \(\PSL_2(\R)\) to
\(\DeltaTri(2,5,\infty)\).  After applying a suitable element of
\(\SL_2(\R)\), we may and shall assume that
\[
\Gamma(X,q)=\DeltaTri(2,5,\infty).
\]
This replacement does not change the Teichm\"uller curve, the
square or nonsquare character of \(q\), or the
\(\GL_2^+(\R)\)-orbit appearing in the equality statement.

\subsection{The square locus}

Let \(X_5\) be the closed Riemann surface obtained by normalizing and
completing
\[
y^2=x^5-1,
\]
and let
\[
\omega_{\mathrm{pent}}:=\frac{\mathrm{d}x}{y}.
\]
Set
\[
\PentagonOrbit
:=\GL_2^+(\R)\cdot(X_5,\omega_{\mathrm{pent}}).
\]
Thus \(\PentagonOrbit\) is the \(\GL_2^+(\R)\)-orbit of the
double-pentagon translation surface in \(\cH_2(2)\).

\begin{proposition}\label{prop:sq-eq}
	Let \(X\) be a closed Riemann surface of genus two and let
	\(\omega\) be a nonzero holomorphic one-form on \(X\).  Set
	\(q=\omega^2\).
	If
	\[
	\Gamma(X,q)=\DeltaTri(2,5,\infty),
	\]
	then
	\[
	(X,\omega)\in\PentagonOrbit.
	\]
\end{proposition}

\begin{proof}
	As in the proof of Proposition~\ref{prop:sq-mod}, the actual
	Veech group of the translation surface \((X,\omega)\) is the full
	inverse image of its projective Veech group:
	\[
	\SL(X,\omega)
	=
	\operatorname{pr}^{-1}
	\bigl(\Gamma(X,\omega^2)\bigr),
	\]
	where
	\[
	\operatorname{pr}:\SL_2(\R)\longrightarrow\PSL_2(\R)
	\]
	is the natural projection.  Indeed, every affine automorphism of
	the half-translation surface \((X,\omega^2)\) is affine for
	\((X,\omega)\), and the hyperelliptic involution has derivative
	\(-I\).
	
	Let \(\bar R\in\DeltaTri(2,5,\infty)\) be an elliptic element of
	order five.  Choose its lift
	\[
	R\in\SL_2(\R)
	\]
	having actual order five.  After replacing \((X,\omega)\) by an
	\(\SL_2(\R)\)-equivalent translation surface, we may assume that
	\(R\in\SO(2)\).  This replacement does not affect the desired
	conclusion, since \(\PentagonOrbit\) is a
	\(\GL_2^+(\R)\)-orbit.
	
	Choose
	\[
	f\in\Aff^+(X,\omega)
	\qquad\text{with}\qquad
	Df=R.
	\]
	Since \(R\in\SO(2)\), the affine automorphism \(f\) is conformal,
	and
	\(f^*\omega=\lambda\omega\)
	for some \(\lambda\in\C^*\).  The automorphism group \(\Aut(X)\) is finite, so
	\(f\) has finite order, say \(N\).  Since the derivative map is a
	homomorphism,
	\[
	(Df)^N=D(f^N)=I.
	\]
	Thus the order of \(Df\) divides \(N\), and hence \(5\mid N\).
	Consequently, the cyclic group \(\langle f\rangle\) contains an automorphism
	\(\sigma\in\Aut(X)\)
	of order five.  As \(\sigma\) is a power of \(f\), the one-form \(\omega\)
	is an eigenform of \(\sigma^*\); that is,
	\(\sigma^*\omega=\mu\omega\) for some \(\mu\in\C^*\).
	
	We next determine the Riemann surface \(X\).  Since \(X\) has genus
	two, its hyperelliptic involution \(h_X\) is central in
	\(\Aut(X)\).  Hence \(\sigma\) descends through the hyperelliptic
	quotient to an automorphism
	\[
	\bar\sigma
	\in
	\Aut\bigl(X/\langle h_X\rangle\bigr)
	\cong\Aut(\mathbb P^1).
	\]
	The kernel of the descent homomorphism is
	\(\langle h_X\rangle\), which has order two.  Hence
	\[
	\langle\sigma\rangle\cap\langle h_X\rangle=\{\id_X\},
	\]
	since \(\sigma\) has order five.  The descent homomorphism is
	therefore injective on \(\langle\sigma\rangle\), and
	\(\bar\sigma\) has order five.
	
	The six branch points of the hyperelliptic map
	\[
	\pi:X\longrightarrow\mathbb P^1
	\]
	are the images of the fixed-point set \(\operatorname{Fix}(h_X)\) of
	\(h_X\).  Since \(\sigma\) commutes with \(h_X\), it preserves this set.
	The identity
	\[
	\pi\circ\sigma=\bar\sigma\circ\pi
	\]
	therefore shows that the six branch points form a
	\(\bar\sigma\)-invariant set.  A nontrivial M\"obius
	transformation of order five has two fixed points, and every other
	orbit has length five.  It follows that the branch set consists of
	one fixed point of \(\bar\sigma\) and one orbit of length five.

	After a M\"obius change of coordinate, we may assume that the two
	fixed points of \(\bar\sigma\) are \(0\) and \(\infty\).  Hence
	\(\bar\sigma(x)=\zeta_5x\)
	for some primitive fifth root of unity \(\zeta_5\).  Interchanging
	\(0\) and \(\infty\), if necessary, we may assume that the fixed
	branch point is \(\infty\).  The remaining five branch points then
	form an orbit
	\[
	\{c,\zeta_5c,\zeta_5^2c,\zeta_5^3c,\zeta_5^4c\}
	\]
	for some \(c\in\C^*\).  After rescaling the coordinate \(x\), we may
	further assume that \(c=1\).  Thus the remaining branch points are
	the fifth roots of unity.  Therefore
	\[
	X\cong X_5,
	\qquad
	X_5:\ y^2=x^5-1.
	\]
	Under this identification, the order-five lift is
	\[
	\sigma(x,y)=(\zeta_5x,y).
	\]
	
	Set
	\[
	\eta_1:=\frac{\mathrm{d}x}{y},
	\qquad
	\eta_2:=\frac{x\,\mathrm{d}x}{y}.
	\]
	These forms constitute a basis of the vector space
	of holomorphic one-forms on \(X_5\), and
	\[
	\sigma^*\eta_1=\zeta_5\eta_1,
	\qquad
	\sigma^*\eta_2=\zeta_5^2\eta_2.
	\]
	Thus \(\eta_1\) and \(\eta_2\) are eigenforms of \(\sigma^*\) with
	distinct eigenvalues.  Since \(\omega\) is also an eigenform, it is a
	scalar multiple of either \(\eta_1\) or \(\eta_2\).
	
	The two eigenforms have divisors
	\[
	\operatorname{div}(\eta_1)=2P_\infty,
	\qquad
	\operatorname{div}(\eta_2)=P_0^++P_0^-,
	\]
	where \(P_0^\pm\) are the two points above \(x=0\).
	Earle and Gardiner identify the corresponding Teichm\"uller
	curves explicitly.  The eigenform \(\eta_1\) gives the
	double-pentagon translation surface, whose projective Veech group
	is \(\DeltaTri(2,5,\infty)\), whereas \(\eta_2\) gives the
	regular-decagon translation surface, whose projective Veech group
	is \(\DeltaTri(5,\infty,\infty)\); see
	\cite[Sections~10--11]{EarleGardiner1997}.  These groups have
	different orbifold signatures, and hence the possibility that
	\(\omega\) is a scalar multiple of \(\eta_2\) is excluded.
	
	Therefore
	\[
	\omega=c\,\frac{\mathrm{d}x}{y}
	\]
	for some \(c\in\C^*\).  Multiplication by \(c\) is the action of
	an element of \(\GL_2^+(\R)\).  Taking into account the initial
	\(\SL_2(\R)\)-change of the translation surface, we conclude that
	\[
	(X,\omega)
	\in
	\GL_2^+(\R)\cdot(X_5,\omega_{\mathrm{pent}})
	=
	\PentagonOrbit.
	\]
\end{proof}

\begin{remark}
  Alternatively, one may combine the description of primitive
  Teichm\"uller curves in genus two with McMullen's theorem that the
  regular pentagon and regular decagon generate the only genus-two
  Teichm\"uller curves of discriminant five
  \cite[Theorem~1.4]{McMullen2005}.  
\end{remark}

\subsection{Two arithmetic nonsquare strata}

\begin{proposition}\label{prop:arith-eq}
  There is no half-translation surface
  \[
    (X,q)\in
    \cQ_2^{\mathrm{ns}}(2,2)\cup\cQ_2(2,1,1)
  \]
  such that
  \(\Gamma(X,q)=\DeltaTri(2,5,\infty)\).
\end{proposition}

\begin{proof}
  By Propositions~\ref{prop:q22-arith} and~\ref{prop:q211-arith},
  every Veech surface in either stratum is
  arithmetic.  The Hecke triangle group
  \(\DeltaTri(2,5,\infty)\) is nonarithmetic
  \cite[Section~1.4]{HubertSchmidt2001}, so it cannot be the projective
  Veech group in either stratum.
\end{proof}

\subsection{The stratum
  \texorpdfstring{\(\cQ_2(1,1,1,1)\)}{Q(1,1,1,1)}}

\begin{proposition}\label{prop:q1111-eq}
  There is no half-translation surface
  \((X,q)\in\cQ_2(1,1,1,1)\) such that
  \[
    \Gamma(X,q)=\DeltaTri(2,5,\infty).
  \]
\end{proposition}

\begin{proof}
  Suppose that such a half-translation surface \((X,q)\) exists.
  Let
  \[
    \pi_{\mathrm{or}}:(Y,\alpha)\longrightarrow(X,q)
  \]
  be the canonical orientation cover, and set
  \[
    K_Y:=\ker\left(
      D:\Aff^+(Y,\alpha;\iota)\longrightarrow\SL_2(\R)
    \right).
  \]
  Lemma~\ref{lem:lift} gives
  \[
    D\bigl(\Aff^+(Y,\alpha;\iota)\bigr)
    =\operatorname{pr}^{-1}\bigl(\DeltaTri(2,5,\infty)\bigr).
  \]
  The inverse image of an elliptic element of projective order five
  contains an element \(R\in\SL_2(\R)\) of actual order five.  Hence
  \[
    \widetilde C:=D^{-1}(\langle R\rangle)
    \subseteq\Aff^+(Y,\alpha;\iota)
  \]
  is a finite group of order \(5|K_Y|\).  By
  Lemma~\ref{lem:kernel}\textup{(ii)}, one has
  \(5\nmid|K_Y|\).  Cauchy's theorem therefore gives an element
  \(F\in\widetilde C\) of order five whose derivative is nontrivial.
  Replacing \(F\) by a suitable power, we may assume that
  \(DF=R\).

  The four double zeros of \(\alpha\) form an \(F\)-invariant set.
  Since a permutation of four points has no element of order five,
  \(F\) fixes all four zeros individually.  Let \(r\) be the number of fixed points of \(F\) and
  let \(g_0\) be the genus of \(Y/\langle F\rangle\).  Since
  \(g(Y)=5\), Riemann--Hurwitz gives
  \[
    8
    =5(2g_0-2)+4r.
  \]
  Here \(r\geq4\).  If \(g_0=0\), the equation gives
  \(4r=18\); if \(g_0=1\), it gives \(r=2\); and if \(g_0\geq2\),
  the right-hand side is at least \(10+16>8\).  All possibilities are
  impossible.
\end{proof}

\begin{proof}[Proof of Theorem~\ref{thm:main}]
  The area bound is Corollary~\ref{cor:bound}.  In the equality case, Lemma~\ref{lem:small-orb} shows that
  \(\Gamma(X,q)\) is conjugate in \(\PSL_2(\R)\) to
  \(\DeltaTri(2,5,\infty)\).  After applying a suitable element of
  \(\SL_2(\R)\), we may assume that
  \[
  \Gamma(X,q)=\DeltaTri(2,5,\infty).
  \]
  Proposition~\ref{prop:sq-eq} identifies the square equality case
  with the \(\GL_2^+(\R)\)-orbit of the double-pentagon translation
  surface.  Propositions~\ref{prop:arith-eq} and
  \ref{prop:q1111-eq} exclude all nonsquare strata.
  
Conversely, by the computation of Earle and Gardiner \cite{EarleGardiner1997}, the
double-pentagon translation surface has projective Veech group
\(\DeltaTri(2,5,\infty)\), and hence its Teichm\"uller curve has
hyperbolic area \(3\pi/5\).
\end{proof}

\appendix

\section{Algebraic proof of Vasilyev's construction}
\label{app:quotients}

The quotient construction originates in Vasilyev's work
\cite[Sections~3--4]{Vasilyev2005}.  For a closed Riemann surface
\(S\), we write \(\C(S)\) for its field of meromorphic functions.
This appendix gives a self-contained algebraic proof, specialized to
genus two, of all the ramification, zero, and equivariance data used
in the main argument.  The algebraic model also describes explicitly
the orientation cover, the Vasilyev quotient, and the induced
involutions. 

Whenever an affine
algebraic model below is singular, the corresponding closed Riemann
surface is understood to be obtained by normalizing and completing
that model.

\subsection{Algebraic normal form}

The classification in Subsection~\ref{sub:nonsquare-strata}
gives the possible zero divisors.  We now choose an algebraic model
that also records their positions relative to the hyperelliptic
involution.

\begin{proposition}
	\label{prop:normal-form}Let \((X,q)\) be a half-translation surface on a closed Riemann
	surface of genus two, and suppose that \(q\) is not a global square.
	There is a Weierstrass point
	\(W_\infty\notin\operatorname{div}(q)\) and an affine plane curve model 
	\begin{equation}\label{eq:x-model}
		y^2=P(x)
	\end{equation}
	of $X$	where \(P\) is squarefree of degree five, \(W_\infty\) is the point at
	infinity, and the hyperelliptic involution is
	\[
	h_X(x,y)=(x,-y).
	\]
	In these coordinates
	\begin{equation}\label{eq:q-model}
		q=R(x)\left(\frac{\mathrm{d}x}{y}\right)^2,
	\end{equation}
where \(R\) is a squarefree polynomial of degree two.
	
	Let
	\[
	D=\gcd(P,R),
	\qquad r=\deg D\in\{0,1,2\}.
	\]
	Then the hyperelliptic positions of the zeros are as follows:
	\begin{center}
		\begin{tabular}{ccl}
			\toprule
			\(r\) & zero pattern of \(q\) & hyperelliptic position of the zeros\\
			\midrule
			\(2\) & \((2,2)\)
			& two distinct Weierstrass points\\
			\(1\) & \((2,1,1)\)
			& one Weierstrass point and one conjugate pair\\
			\(0\) & \((1,1,1,1)\)
			& two conjugate pairs\\
			\bottomrule
		\end{tabular}
	\end{center}
\end{proposition}

\begin{proof}
	A closed Riemann surface of genus two has six Weierstrass points,
	whereas \(\deg\operatorname{div}(q)=4\).  We may therefore choose a
	Weierstrass point \(W_\infty\) outside the zero divisor of \(q\).
	Putting its image under the hyperelliptic quotient at infinity gives
	\eqref{eq:x-model} with \(P\) squarefree of degree five.
	
	In the model \eqref{eq:x-model}, the space
	\(H^0(X,\Omega_X)\) of holomorphic one-forms has basis
	\[
	\frac{\mathrm{d}x}{y},
	\qquad
	\frac{x\,\mathrm{d}x}{y}.
	\]
	Lemma~\ref{lem:mult} shows that the three symmetric products
	\[
	\left(\frac{\mathrm{d}x}{y}\right)^2,
	\qquad
	x\left(\frac{\mathrm{d}x}{y}\right)^2,
	\qquad
	x^2\left(\frac{\mathrm{d}x}{y}\right)^2
	\]
	form a basis of the space
	\(H^0(X,\Omega_X^{\otimes 2})\) of holomorphic quadratic differentials.
	Hence
	\[
	q=R(x)\left(\frac{\mathrm{d}x}{y}\right)^2
	\]
	for a polynomial \(R\) of degree at most two.  If \(t\) is a local
	coordinate of \(X\) at infinity with \(x=t^{-2}\), then
	\[
	\frac{\mathrm{d}x}{y}=u(t)t^2\,\mathrm{d}t,
	\qquad u(0)\neq0.
	\]
	Thus the order of \(q\) at infinity is \(4-2\deg R\).  By the choice
	of \(W_\infty\), the differential \(q\) is nonzero there, and hence
	\(\deg R=2\).
	
	Every holomorphic one-form has the form
	\[
	(a+bx)\frac{\mathrm{d}x}{y}.
	\]
	It follows that \(q\) is a global square if and only if \(R\) is a
	constant multiple of the square of a linear polynomial.  Since \(q\)
	is nonsquare and \(\deg R=2\), the polynomial \(R\) has two distinct
	roots.
	
	Let \(a_i\) be a zero of \(R\).  We distinguish two cases according
	to whether \(a_i\) is a branch value of the hyperelliptic projection
	\(x:X\to\mathbb P^1\).
	
	First suppose that \(P(a_i)\neq0\).  Then the fiber over \(x=a_i\)
	consists of two distinct points
	\[
	p_i^\pm=(a_i,\pm\sqrt{P(a_i)}),
	\]
	which are exchanged by \(h_X\).  Since \(x\) is unramified at these
	points, \(x-a_i\) is a local coordinate, and
	\(\mathrm{d}x/y\) is nonvanishing there.  Because \(R\) has a simple
	zero at \(a_i\), the differential \(q\) has a simple zero at each of
	\(p_i^+\) and \(p_i^-\).
	
	Now suppose that \(P(a_i)=0\).  Then the fiber over \(x=a_i\) consists
	of a single Weierstrass point \(p_i=(a_i,0)\).  Since \(a_i\) is a
	simple zero of \(P\), there is a local coordinate \(s\) centered at
	\(p_i\) such that
	\[
	x-a_i=s^2,
	\qquad
	y=s\,u(s),
	\qquad
	u(0)\neq0.
	\]
	Consequently,
	\[
	\frac{\mathrm{d}x}{y}=\frac{2}{u(s)}\,\mathrm{d}s
	\]
	is nonvanishing at \(p_i\).  Since
	\(R(x)=(x-a_i)\widetilde R(x)\), with
	\(\widetilde R(a_i)\neq0\), the differential \(q\) has a double zero
	at \(p_i\).
	
	Thus a root of \(R\) which is not a branch value of the hyperelliptic
	projection gives a conjugate pair of simple zeros, whereas a common
	root of \(P\) and \(R\) gives a double zero at the corresponding
	Weierstrass point.  Since \(P\) and \(R\) are squarefree,
	\(r=\deg\gcd(P,R)\) counts their common roots.  The three cases
	\(r=2,1,0\) therefore give the zero patterns \((2,2)\),
	\((2,1,1)\), and \((1,1,1,1)\), respectively.
\end{proof}

\subsection{An algebraic model for the orientation cover and the
	Vasilyev quotient}

Fix the normal form
\eqref{eq:x-model}--\eqref{eq:q-model}.  Write
\begin{equation}\label{eq:factor}
	P=DP_0,
	\qquad
	R=DR_0,
	\qquad
	\gcd(P_0,R_0)=1,
\end{equation}
where \(D=\gcd(P,R)\) and \(r=\deg D\).  

\begin{proposition}
	\label{prop:cover-model}
	Let
	\[
	\pi_{\mathrm{or}}:(Y,\alpha)\longrightarrow(X,q)
	\]
	be the connected canonical orientation double cover.  Then \(Y\) is
	the closed Riemann surface obtained by normalizing and completing the
	affine model
	\begin{equation}\label{eq:y-model}
		y^2=P(x),
		\qquad
		z^2=R(x),
	\end{equation}
	and
	\begin{equation}\label{eq:alpha-model}
		\alpha=z\frac{\mathrm{d}x}{y}.
	\end{equation}
	The deck involution of \(\pi_{\mathrm{or}}\) and the unique lift of
	the hyperelliptic involution preserving \(\alpha\) are respectively
	\[
	\iota(x,y,z)=(x,y,-z),
	\qquad
	\tau(x,y,z)=(x,-y,-z).
	\]
	They satisfy
	\[
	\iota^*\alpha=-\alpha,
	\qquad
	\tau^*\alpha=\alpha,
	\qquad
	\iota\circ\tau=\tau\circ\iota.
	\]
	
	Let
	\[
	\pi_V:Y\longrightarrow Z=Y/\langle\tau\rangle.
	\]
	Then \(Z\) is the closed Riemann surface obtained by normalizing and
	completing the affine model
	\begin{equation}\label{eq:z-model}
		v^2=P_0(x)R_0(x),
	\end{equation}
	where
	\[
	v=\frac{yz}{D(x)}
	\]
	in the meromorphic function field of \(Y\).  Moreover,
	\[
	\alpha=\pi_V^*\omega,
	\qquad
	\omega=R_0(x)\frac{\mathrm{d}x}{v}.
	\]
	The involution \(\iota\) descends to
	\[
	j(x,v)=(x,-v),
	\]
	and \(j^*\omega=-\omega\).  Finally,
	\[
	Z/\langle j\rangle
	\cong
	Y/\langle\iota,\tau\rangle
	\cong
	X/\langle h_X\rangle
	\cong\mathbb P^1.
	\]
	In particular, when \(g(Z)\geq2\), the involution \(j\) is the
	hyperelliptic involution of \(Z\).
\end{proposition}

\begin{proof}
	Adjoin to
	\[
	\C(X)=\C(x,y),
	\qquad
	y^2=P(x),
	\]
	an element \(z\) satisfying \(z^2=R(x)\).  Since \(q\) is not a
	global square, \(R\) is not a square in \(\C(X)\).  Hence
	\[
	\C(Y)=\C(x,y,z)
	\]
	is a quadratic extension of \(\C(X)\), and its associated closed
	Riemann surface is the normalization and completion of
	\eqref{eq:y-model}.  The form
	\[
	\alpha=z\frac{\mathrm{d}x}{y}
	\]
	satisfies \(\alpha^2=\pi_{\mathrm{or}}^*q\).  Thus this is the
	connected canonical orientation cover, with deck involution
	\(\iota(x,y,z)=(x,y,-z)\).
	
	The two lifts of \(h_X(x,y)=(x,-y)\) are
	\[
	(x,y,z)\longmapsto(x,-y,z)
	\quad\text{and}\quad
	(x,y,z)\longmapsto(x,-y,-z).
	\]
	They act on \(\alpha\) with opposite signs.  Therefore the unique lift
	preserving \(\alpha\) is
	\(\tau(x,y,z)=(x,-y,-z)\), and the displayed formulas give
	\(\iota\tau=\tau\iota\).
	
	Set \(v=yz/D(x)\).  Then
	\[
	v^2=\frac{P(x)R(x)}{D(x)^2}=P_0(x)R_0(x).
	\]
	Both \(x\) and \(v\) are fixed by \(\tau\).  Moreover,
	\[
	\C(Y)=\C(x,v)(y),
	\qquad
	y^2=P(x),
	\]
	and \(\tau(y)=-y\).  Hence
	\[
	[\C(Y):\C(x,v)]=2,
	\qquad
	\C(Y)^{\langle\tau\rangle}=\C(x,v).
	\]
	It follows that \(Z=Y/\langle\tau\rangle\) is the normalization and
	completion of \eqref{eq:z-model}.  The one-form
	\[
	\omega=R_0(x)\frac{\mathrm{d}x}{v}
	\]
	satisfies
	\[
	\pi_V^*\omega
	=R_0(x)\frac{\mathrm{d}x}{yz/D(x)}
	=z\frac{\mathrm{d}x}{y}
	=\alpha.
	\]
	Since \(\iota(v)=-v\), the involution \(\iota\) descends to
	\(j(x,v)=(x,-v)\), and \(j^*\omega=-\omega\).  Since \(\iota\) and
	\(\tau\) commute, quotienting first by \(\tau\) and then by the
	induced involution \(j\) is equivalent to quotienting \(Y\) by both
	involutions.  Therefore
	\[
	Z/\langle j\rangle
	\cong
	Y/\langle\iota,\tau\rangle
	\cong
	X/\langle h_X\rangle
	\cong\mathbb P^1.
	\]
\end{proof}

\subsection{Proof of Proposition~\ref{prop:quot-data}}

\begin{proof}[Proof of Proposition~\ref{prop:quot-data}]
	Retain the notation
	\[
	P=DP_0,
	\qquad
	R=DR_0,
	\qquad
	r=\deg D
	\]
	from \eqref{eq:factor}.  By Proposition~\ref{prop:normal-form}, the
	cases \(r=2,1,0\) correspond respectively to
	\[
	\cQ_2^{\mathrm{ns}}(2,2),
	\qquad
	\cQ_2(2,1,1),
	\qquad
	\cQ_2(1,1,1,1).
	\]
	
	The differential \(q\) has \(4-2r\) odd-order zeros.  The orientation
	cover \(\pi_{\mathrm{or}}:Y\to X\) has degree two and is simply
	ramified precisely over these points.  Since \(g(X)=2\),
	Riemann--Hurwitz gives \(g(Y)=5-r\).
	
	The projection \(x:Z\to\mathbb P^1\) also has degree two.  The
	polynomial \(P_0R_0\) is squarefree of degree \(7-2r\), so the branch
	values are its \(7-2r\) finite roots together with infinity.  Hence
	Riemann--Hurwitz gives \(g(Z)=3-r\).
	
	We next determine the branch values of
	\(\pi_V:Y\to Z\).  At infinity, take a local coordinate \(t\) on
	\(X\) with \(x=t^{-2}\).  The two points of \(Y\) above infinity may
	be distinguished by a sign \(\varepsilon\in\{\pm1\}\), and locally
	one may write
	\[
	y=t^{-5}u(t^2),
	\qquad
	z=\varepsilon t^{-2}w(t^2),
	\qquad
	u(0)w(0)\neq0.
	\]
	The involution \(\tau\) sends
	\((t,\varepsilon)\) to \((-t,-\varepsilon)\), and therefore exchanges
	the two points above infinity.  Thus \(\pi_V\) is unramified there.
	
	An affine fixed point of \(\tau\) must satisfy \(y=z=0\), and hence
	lies over a common root \(a\) of \(P\) and \(R\).  Put \(s=x-a\) and
	write
	\[
	P=s\,p(s),
	\qquad
	R=s\,r_1(s),
	\qquad
	p(0)r_1(0)\neq0.
	\]
	On the normalization of \(Y\), the quotient \(z/y\) has the two
	holomorphic values obtained from the two square roots of
	\(r_1(s)/p(s)\).  Thus the affine point \(x=a,\ y=z=0\) separates
	into two points of \(Y\).  On each branch, \(y\) is a local
	coordinate and \(\tau\) acts by \(y\mapsto-y\).  Each of the two
	points is therefore a simple ramification point of \(\pi_V\).
	
	There are no other fixed points of \(\tau\).  Consequently, each
	common root of \(P\) and \(R\) gives two simple ramification points,
	and \(\pi_V\) has exactly \(2r\) such points.
	
	At the same value \(x=a\), the equation
	\[
	v^2=P_0(x)R_0(x)
	\]
	gives two distinct points of \(Z\), because
	\(P_0(a)R_0(a)\neq0\).  These are the corresponding branch values,
	and they are interchanged by \(j\).  Moreover,
	\(\omega=R_0(x)\,\mathrm{d}x/v\) is nonvanishing at both points.
	Thus \(B_V\) consists of \(2r\) regular points.  For \(r=2\), these
	are four points on which \(j\) acts without fixed points; for \(r=1\),
	they are two points interchanged by \(j\); and for \(r=0\), one has
	\(B_V=\varnothing\).
	
	It remains to determine the zeros of \(\omega\).  Let \(b\) be a
	root of \(R_0\).  Since \(b\) is not a root of \(P_0\), the point of
	\(Z\) above \(x=b\) is a branch point of
	\(x:Z\to\mathbb P^1\).  In a local coordinate \(t\) with
	\(x-b=t^2\), the form \(\mathrm{d}x/v\) is nonvanishing, whereas
	\(R_0(x)\) has order two.  Hence \(\omega\) has a double zero at this
	point.
	
	These are all the finite zeros.  At infinity, with \(x=t^{-2}\), the
	functions \(v\) and \(R_0\) have pole orders \(7-2r\) and \(4-2r\),
	respectively, while \(\mathrm{d}x\) has pole order three.  It follows
	that
	\[
	\omega=R_0(x)\frac{\mathrm{d}x}{v}
	\]
	is nonvanishing at infinity.  Thus \(\omega\) has precisely
	\(2-r\) zeros, each of order two.
	
	When \(r=2\), the surface \(Z\) has genus one and \(\omega\) has no
	zeros, so \((Z,\omega)\) is a flat torus.  When \(r=1\), one has
	\((Z,\omega)\in\cH_2(2)\).  When \(r=0\), one has
	\((Z,\omega)\in\cH_3(2,2)\).  Since the zeros of \(\omega\) are fixed
	by \(j\), they are Weierstrass points whenever \(g(Z)\geq2\).
	
	Finally, suppose that \(r=0\).  If \(b\) is a root of \(R\), the two
	points
	\[
	(b,\sqrt{P(b)},0),
	\qquad
	(b,-\sqrt{P(b)},0)
	\]
	are zeros of \(\alpha=z\,\mathrm{d}x/y\), and \(\tau\) exchanges
	them.  Since \(R\) has two roots, \(\tau\) acts on the four zeros of
	\(\alpha\) as a product of two disjoint transpositions.
\end{proof}

\end{document}